\documentclass[final]{siamart190516}

\usepackage{algorithm}
\usepackage[noend]{algpseudocode}
\makeatletter
\def\BState{\State\hskip-\ALG@thistlm}
\usepackage{empheq}         
\usepackage{mathtools}      
\usepackage{url}            
\usepackage[T1]{fontenc}    
\usepackage[utf8]{inputenc} 
\usepackage{numprint}       
\usepackage{amsmath}        
\usepackage{mathrsfs}       
\usepackage{amssymb}        
\usepackage{amsfonts}       
\usepackage{enumitem}
\usepackage{xcolor}
\usepackage{ulem}

\usepackage[title]{appendix}
\usepackage{setspace}

\newcommand{\Rbb}{\mathbb{R}}
\newcommand{\pa}[1]{\left(#1\right)}
\newcommand{\pac}[1]{\left[#1\right]}
\newcommand{\paa}[1]{\left\{#1\right\}}

\newcommand{\aheader}[3]{[$#1$] = #2($#3$)} 

\newcommand{\dd}{\text{d}}
\newcommand{\ddt}[2]{\frac{\dd #1}{\dd #2}}
\newcommand{\Ddtz}[2]{\frac{\D }{\dd #2} #1\Bigr\vert_{#2 = 0}}
\newcommand{\ddp}[2]{\frac{\partial #1}{\partial #2}}
\newcommand{\Dndt}[3]{\frac{\D^#3}{\dd #2^#3} #1}
\newcommand{\Dndtz}[3]{\frac{\D^#3}{\dd #2^#3}#1\Bigr\vert_{#2 = 0}}
\newcommand{\scalp}[2]{\left\langle #1,#2 \right\rangle}

\newcommand{\acos}{\operatorname{acos}}

\newcommand{\RN}[1]{\textup{\uppercase\expandafter{\romannumeral#1}}}

\newcommand{\map}[5]{ #1 \,:\, #2& \to #3 \\ #4& \mapsto #5}

\newcommand{\matr}[2]{\Rbb^{#1\times#2}}
\newcommand{\D}{\operatorname{D}\hspace{-0.08cm}}

\newcommand{\mcal}{\mathcal{M}}

\newcommand{\grad}[1]{\operatorname{grad}\hspace{-0.08cm}#1}
\newcommand{\hessx}[1]{\operatorname{Hess}\hspace{-0.08cm}#1}

\newcommand{\tcal}{\mathcal{T}}
\newcommand{\diff}[3]{\operatorname{D}\hspace{-0.08cm}#1\pa{#2}\pac{#3}}

\newcommand{\trace}{\operatorname{Tr}}
\newcommand{\Ocal}[1]{O\pa{#1}}
\newcommand{\ocal}[1]{o\pa{#1}}

\newcommand{\SPD}[1]{\mathbb{S}^{#1}_+}
\newcommand{\SYM}[1]{\mathbb{S}^{#1}}
 
\newcommand{\ucal}{\mathcal{U}}
\newcommand{\ncal}{\mathcal{N}}
\newcommand{\rank}[1]{\operatorname{rank}\pa{#1}}
\newcommand{\hF}{\hat F}
\newcommand{\hH}{\hat H}
\newcommand{\hA}{\hat A}
\newcommand{\hx}{{\hat x}}
\newcommand{\hy}{{\hat y}}
\newcommand{\hw}{{\hat w}}
\newcommand{\htt}{\hat t}
\newcommand{\hn}{\hat n}

\newcommand{\hz}{\hat z}
\newcommand{\hR}{\hat R}
\newcommand{\hG}{\hat G}
\newcommand{\hTcal}{\hat \tcal}
\newcommand{\hUcal}{\hat \ucal}

\title{Continuation Methods for Riemannian Optimization}
\author{Axel S\'{e}guin\thanks{Institute of Mathematics, EPF Lausanne, 1015 Lausanne, Switzerland (axel.seguin@epfl.ch, daniel.kressner@epfl.ch)} \and Daniel Kressner\footnotemark[1] }
\date{}
\begin{document}
	\maketitle
	\begin{abstract}
		Numerical continuation in the context of optimization can be used to mitigate convergence issues due to a poor initial guess. In this work, we extend this idea to Riemannian optimization problems, that is, the minimization of a target function on a Riemannian manifold. For this purpose, a suitable homotopy is constructed between the original problem and a problem that admits an easy solution. We develop and analyze a path-following numerical continuation algorithm on manifolds for solving the resulting parameter-dependent problem. To illustrate our developments, we consider two classical applications of Riemannian optimization: the computation of the Karcher mean and low-rank matrix completion. We demonstrate that numerical continuation can yield improvements for challenging instances of both problems.
	\end{abstract}
	
	\section{Introduction}\label{s.fromHomotopyToOptimization}
	\normalem
	
	This work aims at developing and analyzing numerical continuation for Riemannian optimization. Let us first recall the setting of numerical continuation for nonlinear equations. Given a nonlinear equation
	\begin{equation}\label{eq.EuclideanEquation}
		F(x) = 0,
	\end{equation}
	for a smooth function $F:\Rbb^d\to\Rbb^d$, numerical continuation~\cite{allg, deuflhard} is used to track solutions of~\eqref{eq.EuclideanEquation} when the problem is smoothly perturbed. This can be useful for, e.g., ensuring global convergence of the Newton method by progressively transforming a simple problem with a known solution into~\eqref{eq.EuclideanEquation}. More specifically, one considers a parametrized family of equations, 
	\begin{equation}\label{eq.eculdieanParametricProblem}
		G(x,\lambda) = 0, \quad\forall\lambda\in\pac{0,1},
	\end{equation}
	with $G:\Rbb^d\times\pac{0,1} \to \Rbb^d$ such that $G(x,1) = F(x)$ holds and a solution $x_0\in\Rbb^d$ of $G(x_0,0) = 0$ can be easily determined.
	The function $G$ is also known as a \emph{homotopy}. Under suitable assumptions, the solution set 
	\begin{equation}
		G^{-1}(0) = \paa{(x,\lambda)\in \Rbb^d\times\pac{0,1} : G(x,\lambda) = 0}
	\end{equation}
	to the parametric problem~\eqref{eq.eculdieanParametricProblem} contains a smooth $x(\lambda)$, $\lambda\in\pac{0,1}$, connecting $x_1 = x(1)$, the solution to the original problem, to $x_0 = x(0)$.
	
	Homotopy methods are also relevant in optimization.
	Optimization methods for a given target function $f:\mathbb R^n \to \mathbb R$ often aim at retrieving critical points, that is, solutions to
	\begin{equation}
		F(x) = \nabla f (x) = 0.
	\end{equation}
	Homotopy methods can be useful for, e.g., ensuring global convergence (to a critical point) by tracking critical points of a parametrized optimization problem, which amounts to considering
	\begin{equation}\label{eq.EuclideanParametricCriticalPoint}
		G(x,\lambda) = \nabla f(x,\lambda) = 0, \quad\forall\lambda\in\pac{0,1}.
	\end{equation}
	This approach to optimization problems has been widely studied in the literature, both for unconstrained and constrained optimization problems~\cite{deformationOfLinearPrograms1,deformationOfLinearPrograms2}. Among others, this has led to almost always globally convergent methods for non convex optimization~\cite{prob1homotopies} and nonlinear programming~\cite{deformationOfLinearPrograms2,prob1homotopies2}. Another use of homotopy methods is to improve the convergence behavior of a method by, e.g. defining a homotopy in which a regularization term is reduced progressively~\cite{sparsereghomotopy}. 
	
	Riemannian optimization~\cite{absilbook, boumalBook} is concerned with optimizing a target function $f:\mcal \to \Rbb$ on a smooth manifold $\mcal$ equipped with a Riemannian metric. The geometry of $\mcal$ gives the tools to design optimization methods that produce the iterates guaranteed to stay  on the manifold.  
	
	The Riemannian counterpart of the homotopy~\eqref{eq.EuclideanParametricCriticalPoint} is
	\begin{equation}\label{eq.riemannianParametricCriticalPoint}
		\grad f(x,\lambda) = 0, \quad\forall\lambda\pac{0,1},
	\end{equation}
	where $f: \mcal\times\pac{0,1}\to\Rbb$ and $\grad f(x,\lambda)$ denotes the Riemannian gradient of $f(\cdot,\lambda)$ at $x$.  Continuation methods for~\eqref{eq.riemannianParametricCriticalPoint} need to ensure that $x$ stays on $\mcal$. In this work, we use tools from Riemannian optimization to design  path-following algorithms achieving this demand. A related question has been explored in the more restricted setting of time-varying convex optimization on Hadamard manifolds~\cite{maass}, making use of the exponential map.
	In~\cite{manton}, a theoretical study of parameter-dependent Riemannian optimization is performed; the resulting homotopy-based algorithm involves local charts in order to utilize standard continuation algorithms on Euclidean spaces. In this work, we develop continuation methods within the framework of Riemannian optimization as presented in~\cite{absilbook}, which allows for the convenient design of efficient numerical methods in a general setting.
	
	\paragraph{Outline} After recalling in Section~\ref{s.Euclideancontinuation} the general structure of a path-following predictor-corrector continuation algorithm for nonlinear equations on Euclidean spaces, we introduce in Section~\ref{s.riemannianContinuation} the setting of parametric Riemannian optimization and provide sufficient conditions for the numerical continuation problem to be well-posed. We then translate to the Riemannian setting the predictor-corrector algorithm to address them. We analyse the prediction phase, a key step of the algorithm and also propose a step size adaptivity strategy. Finally, Sections~\ref{s.karcherMean} and~\ref{s.matrixCompletion} are dedicated to the application of the algorithm to two classical Riemannian optimization problems, respectively the computation of the Karcher mean and the low-rank matrix completion problem.

	\section{Euclidean predictor-corrector continuation}\label{s.Euclideancontinuation}
	
	To motivate our Riemannian continuation algorithm, let us first recall the standard predictor-corrector continuation approach; see, e.g.~\cite[chapter 2]{allg}. 
	
	Considering the parametric nonlinear equation~\eqref{eq.eculdieanParametricProblem}, let us assume that $0$ is a regular value of $G$, that is, the differential 
	\begin{equation*}
		\D G(x,\lambda) = \pac{G_x(x,\lambda)\lvert G_\lambda(x,\lambda)}\in\matr{d}{d+1},
	\end{equation*}
	has full rank for each $\pa{x,\lambda}\in G^{-1}(0)$. Then the constant-rank level set theorem~\cite[Theorem 5.12]{lee} asserts the set $G^{-1}(0)$ is an embedded submanifold of $\Rbb^{d+1}$ of dimension $1$ or, in other words, the union of disjoint curves.
	Under the stronger assumption that $G_x(x,\lambda)\in\matr{d}{d}$ has full rank, the implicit function theorem~\cite[Theorem 1.3.1]{implicitFunctionTheorem} implies that it is possible to parametrize each solution curve as a function $x(\lambda)$. Moreover, its derivative is given by 
	\begin{equation}\label{eq.tangentVector}
		x'(\lambda) = - G_x(x(\lambda),\lambda)^{-1}\pac{G_\lambda(x(\lambda),\lambda)}.
	\end{equation}
	In turn, the solution curve in~\eqref{eq.eculdieanParametricProblem} can be obtained from solving the following implicit ODE:
	\begin{equation} \label{eq:davidenko}
		\begin{cases*}
			G_x(x,\lambda)\pac{x'} + G_\lambda(x,\lambda) = 0,\quad \forall\lambda\in\pac{0,1}\\
			x(0) = x_0.
		\end{cases*}
	\end{equation}
	This equation is sometimes called Davidenko equation~\cite{davidenko}. The path-following approach consists of   numerically integrating~\eqref{eq:davidenko} from time $\lambda = 0$ to $\lambda = 1$. The existence of the solution to~\eqref{eq:davidenko} is discussed in~\cite[Theorem 4.2.1]{implicitFunctionTheorem}; see also Theorem~\ref{teo.curveExistence} below.

	Given an approximation $x_k\simeq x(\lambda_k)$ of the solution curve at point $\lambda_k$, a predictor-corrector continuation algorithm first performs a prediction step, which obtains a possibly very rough estimate $y_{k+1}$ of the solution curve at the next point $\lambda_{k+1}$. This is followed by a correction phase which aims at projecting this estimate back to the solution curve.
	
	The most common choices for the \emph{prediction step} are:
	\begin{align}
		\text{\textit{classical prediction}}\: &: \: y_{k+1} = x_k\label{eq.EuclideanClassicalPred}\\
		\text{\textit{tangential prediction}}\: &: \: y_{k+1} = x_k + (\lambda_{k+1}-\lambda_k) t(x_k,\lambda_k)\label{eq.EuclideanTangentialPred},
	\end{align}
	where the tangent vector $t(x_k,\lambda_k):=x^\prime(\lambda_k)$ is obtained from~\eqref{eq.tangentVector}. This requires the solution of a linear system, a cost that is offset by increased prediction accuracy, see~\cite[p.238-239]{deuflhard} and Section~\ref{ss.predictionOrder}. Note that~\eqref{eq.EuclideanTangentialPred} coincides with one step of the Euler method applied to~\eqref{eq:davidenko}. 
	
	In the \textit{correction phase}, the refinement of the estimate $y_{k+1}$ is performed by applying a nonlinear equation solver, typically a Newton-type method, on the equation $G(x,\lambda_{k+1}) = 0$ with initial guess $y_{k+1}$. A sufficiently small step size $\lambda_{k+1} - \lambda_k$ leads to a prediction that is accurate enough to yield (very) fast convergence. Various step size selection strategies have been developed in the literature, see~\cite{allg, deuflhard} and Section~\ref{ss.stepSizeAdaptivity}.
	
	\section{Continuation for Riemannian optimization}\label{s.riemannianContinuation}
	
	In this section, we consider a Riemannian optimization problem depending on a scalar parameter. The parameter can be intrinsic to the problem (e.g., time) or has been artificially added to form a homotopy. Examples of homotopies for Riemannian optimization problems will be given in Sections~\ref{s.karcherMean} and~\ref{s.matrixCompletion}. 
	
	\subsection{Riemannian Davidenko equation}
	
	We consider a $d$-dimensional Riemannian manifold $\mcal$ endowed with the Riemannian metric $\scalp{\cdot}{\cdot}$ and let $\nabla$ denote the Riemannian connection. 
	The parameter-dependent objective function 
	\begin{align*}
		\map{f}{\mcal\times\pac{0,1}}{\Rbb}{(x,\lambda)}{f(x,\lambda)} 
	\end{align*}
	is assumed to be smooth in both arguments (at least of class $C^2$).
	
	For fixed $\lambda \in [0,1]$,
	the Riemannian gradient $\grad{f}(x,\lambda)$ of $f(\cdot,\lambda)$ at $x\in\mcal$
	is defined to be the vector in the tangent space $T_x\mcal$ satisfying 
	\begin{equation*}
		\diff{f}{x,\lambda}{\xi} :=
		\ddt{f(\gamma_{x,\xi}(t),\lambda)}{t}{\big\vert}_{t=0} = \scalp{\grad{f}(x,\lambda)}{\xi}_x, \quad \forall  \xi\in T_x\mcal,
	\end{equation*}
	where $\gamma_{x,\xi}$ is a manifold curve of $\mcal$ such that $\gamma_{x,\xi}(0) = x$ and $\dot\gamma_{x,\xi}(0) = \xi$. Likewise, the Riemannian Hessian $\hessx{f}(x,\lambda)$ of $f(\cdot,\lambda)$ at $x\in\mcal$ is the linear map on the tangent space $T_x\mcal$ satisfying 
	\begin{equation*}
		\hessx{f}(x,\lambda)\pac{\xi} = \nabla_\xi\grad{f}(x,\lambda), \quad \forall  \xi\in T_x\mcal.
	\end{equation*}
	
	Consider the numerical continuation problem~\eqref{eq.riemannianParametricCriticalPoint} of tracking critical points of the  objective function as the parameter $\lambda$ varies. Theorem~\ref{teo.curveExistence} below is inspired by~\cite[Theorem 4.2.1]{implicitFunctionTheorem} and gives sufficient conditions for the existence and parametrizability with respect to $\lambda$ of a differentiable manifold curve joining a critical point $x_0\in\mcal$ at $\lambda = 0$ and a critical point at $\lambda = 1$. Note that we let $B(x_0,L):=\{x\in M\colon d(x_0,x) < L\}$ denote a ball on the manifold, where $d(\cdot,\cdot)$ is the manifold distance function induced by the metric. We recall the manifold distance function is defined as
	\begin{equation}\label{eq.distanceFunction}
		d(x,y) = \underset{\gamma\in\Gamma_{xy}}{\inf}\paa{L(\gamma)}
	\end{equation} 
	where $\Gamma_{xy} = \paa{\gamma:\pac{0,1}\to\mcal\,:\, \gamma(0) = x, \,\gamma(1) = y}$ is the set of piecewise smooth curves joining $x$ and $y$ and $L(\gamma) = \int_{0}^{1}\|\gamma'(\tau)\|_{\gamma(\tau)}d\tau$ is the length of the curve.
	
	For the purpose of the analysis, we will assume that $\mcal$ is \emph{complete}.
	
	\begin{theorem}\label{teo.curveExistence}
		Let $\mcal$ be a complete Riemannian manifold, $\ucal$ be an open subset of $\mcal$ and $V$ an open subset of $\mcal\times \Rbb$ such that ${\ucal \times\pac{0,1}\subset V}$. Consider a scalar field $f\in C^2(V,\Rbb)$. 
		Assume that there exist $x_0\in \ucal$ such that $\grad{f}(x_0,0) = 0$ and a constant $L>0$ such that $B(x_0,L)\subseteq \ucal$. Moreover, suppose that for every $(z,\lambda)\in \ucal \times\pac{0,1}$ it holds that
		\begin{enumerate}[label=(\roman*)]	
			\item $\rank{\hessx{f}(z,\lambda)} = d$, \label{hyp.hessFullRank}
			\item $\|\hessx{f}(z,\lambda)^{-1}\pac{\ddp{}{\lambda}\grad{f}(z,\lambda)}\|_z<L$. \label{hyp.boundedDerivative} 
		\end{enumerate}
		Then there exist an open interval $J \supset [0,1]$ and a curve $x \in C^1(J,\mcal)$ verifying
		\begin{equation}\label{eq.criticalPointsCurve}
			x(0) = x_0, \quad \grad{f}(x(\lambda),\lambda)= 0,\quad\forall \lambda\in\pac{0,1}.
		\end{equation} This curve satisfies the initial value problem
		\begin{equation}\label{eq.riemannianDavidenko}
			\begin{cases*}
				\hessx{f}(x(\lambda),\lambda)\pac{ \dot x(\lambda)} + \ddp{\grad f(x(\lambda),\lambda)}{\lambda} = 0,\quad \forall\lambda\in\pac{0,1},\\
				x(0) = x_0.
			\end{cases*}
		\end{equation}
	\end{theorem}
	Hypothesis~\ref{hyp.hessFullRank} guarantees the parametrizability with respect to $\lambda$ by ensuring the implicit ODE~\eqref{eq.riemannianDavidenko} is well-defined. For a fixed $\lambda$, it is an analogous assumption guaranteeing local quadratic convergence of the Riemannian Newton method~\cite[Theorem 6.3.2]{absilbook}. Hypothesis~\ref{hyp.boundedDerivative} ensures that the manifold curve can be parametrized up to $\lambda = 1$ as the limit point of the curve for $\lambda\to \lambda^*$, for any $0<\lambda^*<1$, is guaranteed to stay in the region $\ucal$ where the Riemannian Hessian is still of full rank. These hypotheses are global a priori assumption that are difficult to verify in practice. Yet, for a large class of problems it is reasonable to assume the Riemannian Hessian is of full rank at the starting point $(x_0,0)$, and therefore the solution curve is at least parametrizable on a possibly smaller interval $\pac{0,\tau}\subseteq \pac{0,1}$. In the following, we call the initial value problem~\eqref{eq.riemannianDavidenko} the \emph{Riemannian Davidenko equation}.  Note that by Hypothesis~\ref{hyp.hessFullRank}, if $x_0$ is a local minimum, then the solution curve to the Riemannian Davidenko equation is a manifold curve of local minima. If we further assume the objective function to be geodesically convex~\cite[Chapter 11]{boumalBook} for each $\lambda\in\pac{0,1}$, this implies that the solution curve consists of global minima.
	
	The following proof of Theorem~\ref{teo.curveExistence} is an adaptation of the proof for the Euclidean case~\cite[Theorem 4.2.1]{implicitFunctionTheorem}. 
	\begin{proof}(of Theorem~\ref{teo.curveExistence}) 
		Consider a local chart $\varphi: \ncal \to \Rbb^d$ such that $x_0\in\ncal$ and ${\ncal\times\pac{0,1}\subseteq V}$. We give a local coordinate representation of the gradient vector field through this local chart by defining 
		\begin{equation*}
			F(\hx,\lambda) := \D\varphi(\varphi^{-1}\pa{\hx})\pac{\grad{f}(\varphi^{-1}(\hx),\lambda)},\quad\forall (\hx,\lambda)\in\varphi\pa{\ncal}\times\pac{0,1}.
		\end{equation*} 
		The Jacobian of this vector field along the vector $\hat v\in\Rbb^d$ is 
		\begin{align*}
			\D_{\hx} F(\hx,\lambda)\pac{\hat v} &= \D^2\varphi(\varphi^{-1}(\hx))\pac{\grad{f}(\varphi^{-1}(\hx),\lambda),\D\varphi^{-1}(\hx)\pac{\hat v}}\\
			&+\D\varphi(\varphi^{-1}(\hx))\pac{\nabla_{\D\varphi^{-1}(\hx)\pac{\hat v}}\grad{f}(\varphi^{-1}(\hx),\lambda)}.
		\end{align*}
		Letting $\hx_0 = \varphi(x_0)$, we find
		\begin{align*}
			F(\hx_0,0) = 0,\qquad
			\D_{\hx} F(\hx_0,0) = \D\varphi(\varphi^{-1}(\hx_0))\circ{\hessx{f}(\varphi^{-1}(\hx_0),\lambda)}\circ\D\varphi^{-1}(\hx_0).
		\end{align*}
		Since local charts are diffeormorphisms, hypothesis (i) implies that $\D_\hx F(\hx_0,0)$ has full rank $d$. Then by applying the implicit function theorem to $F$ at $(\hx_0,0)$ there exist an open interval $I$ containing $0$
		and $\hx \in C^1(I,\varphi(\ncal))$ such that 
		\begin{align*}
			\hx(0) = \hx_0,\qquad 
			F(\hx(\lambda),\lambda) = 0, \quad\forall\lambda\in I.
		\end{align*}
		Defining $x(\lambda) := \varphi^{-1}(\hx(\lambda))$ for $\lambda \in I$, it holds that $x(0) = x_0$. Moreover, there exists 
		$\lambda_0>0$ such that 
		\begin{enumerate}[label=(\arabic*)]
			\item $x$ is defined on $\left[0, \lambda_0\right)$,
			\item $\grad{f}(x(\lambda),\lambda) = 0\quad\forall\lambda\in\left[0, \lambda_0\right)$,
			\item $x$ is continuously differentiable on $\left[0, \lambda_0\right)$,
			\item $x(\lambda)\in \ucal, \quad\forall\lambda\in\left[0, \lambda_0\right)$.
		\end{enumerate}
		Define the following 
		\begin{equation*}
			\lambda^* = \sup\paa{\lambda_0: \text{there exists $x$ such that (1), (2), (3) and (4) are verified}}.
		\end{equation*}
		By the discussion above, $\lambda^*>0$.
		If $\lambda^*>1$, the result is proved. Therefore assume that $0<\lambda^*\le 1$. Differentiation with respect to $\lambda$ of (2) yields
		\begin{equation*}
			x'(\lambda) =- \hessx{f}(x(\lambda),\lambda)^{-1}\pac{\ddp{}{\lambda}\grad{f}(x(\lambda),\lambda)},\quad\forall\lambda\in\left[0, \lambda^*\right).
		\end{equation*} 
		Due to condition (4) and hypothesis $(ii)$ we have $\|x'(\lambda)\|_{x(\lambda)}<L$ for every $\lambda\in\left[0, \lambda^*\right)$.
		This implies 
		\begin{align} \label{eq:inequalityL}
			\tilde L:= \lim_{\lambda\uparrow\lambda^*} d(x_0,x(\lambda))\leq  \lim_{\lambda\uparrow\lambda^*} \int_{0}^{\lambda}\|x'(\tau)\|_{x(\tau)}\dd\tau< \lim_{\lambda\uparrow\lambda^*} \int_{0}^{\lambda}L\dd\tau\leq L.
		\end{align}
		Given a sequence $\{\lambda_k\}$ with $\lambda_k \to \lambda^*$, it follows in an analogous fashion that $\{x(\lambda_k)\}$ is a Cauchy sequence. Because of~\eqref{eq:inequalityL}, 
		$\{x(\lambda_k)\}$ is contained in the closed ball 
		$\overline{B(x_0,\tilde L)} \subset \ucal$ and therefore converges to some $x^* \in \ucal$.

		Now, using a local chart $\psi: \ncal' \to \Rbb^d$
		such that $x^*\in\ncal'$ we can apply the implicit function theorem to
		\begin{equation*}
			\tilde F(\hz,\lambda) = \D\psi(\psi^{-1}\pa{\hz})\pac{\grad{f}(\psi^{-1}(\hz),\lambda)}
		\end{equation*} 
		at $(\psi(x^*),\lambda^*)$ and thus extend $x(\lambda)$ to a larger interval. This contradicts the definition of $\lambda^*$.
	\end{proof}
	
	\subsection{Riemannian predictor-corrector continuation}
	
	The Riemannian \linebreak predictor-corrector continuation algorithm mimics the Euclidean version from Section~\ref{s.Euclideancontinuation} by numerically integrating the Riemannian Davidenko equation~\eqref{eq.riemannianDavidenko}. For the moment, we consider $N$ steps with fixed step size $h_k = 1 / N$, for $k = 1,\dots,N$. A suitable adaptive step size strategy will be discussed in Section~\ref{ss.stepSizeAdaptivity}.
	
	\paragraph{Prediction}
	The classical continuation scheme~\eqref{eq.EuclideanClassicalPred} can be trivially extended to the Riemannian case without any adjustment. The initial guess for the subsequent correction phase is simply
	\begin{equation}\label{eq.classicalPred}
		y_{k+1} = x_k,
	\end{equation}
	the iterate at the previous step of the algorithm. 
	
	The Riemannian extension of the tangential prediction strategy~\eqref{eq.EuclideanTangentialPred} is more involved. It consists of performing a step in the direction of the tangent vector of the solution curve. This tangent vector   can be computed from the Davidenko equation as
	\begin{equation}\label{eq.predictionDirection}
		t(x_k,\lambda_k) := -\hessx{f}(x_k,\lambda_k)^{-1}\pac{ \ddp{\grad f(x_k,\lambda_k)}{\lambda}} \in T_{x_k}\mcal.
	\end{equation}
	We note that this involves the solution of a linear system with the Riemannian Hessian. If its solution by a direct solver (e.g., via the Cholesky decomposition) is too expensive, especially for 
	manifolds of higher dimension, matrix-free Krylov type methods~\cite[chapter 5]{krylovMethodBook} can be used instead.
	
	In the Euclidean case, a tangent vector was simply added to the current iterate. In the manifold setting, this needs to be combined with a retraction in order to make sure that the result is again on the manifold.
	A retraction is a smooth mapping $R\,:\, T\mathcal{M}\to\mathcal{M}$ with the following two properties:
	\begin{itemize}
		\item[1)] $R_x(0_x) = x$, where $0_x$ is the zero element of $T_x\mathcal{M}$ and $R_x$ denotes the restriction of $R$ to $T_x\mathcal{M}$; 
		\item[2)] $\D R_x(0_x) = \operatorname{Id}_{T_x\mathcal{M}}$, where we have identified $T_{0_x}T_x\mathcal{M}\simeq T_x\mathcal{M}$ and $\operatorname{Id}_{T_x\mathcal{M}}$ is the identity mapping on $T_x\mathcal{M}$. 
	\end{itemize}
	These properties ensure that the retraction is a first order approximation of the Riemannian exponential map~\cite[section 10.2]{boumalBook}; the second property is also known as  \emph{local rigidity}.
	More details can be found in~\cite[Chapter 4]{absilbook}; see also Sections~\ref{s.karcherMean} and~\ref{s.matrixCompletion} for examples. 
	The Riemannian tangential prediction step is defined as 
	\begin{equation}\label{eq.tangentialPred}
		y_{k+1} = R_{x_k}(h_k t(x_k,\lambda_k)),
	\end{equation}
	where we recall that $h_k$ denotes the step size.
	
	In the case of a manifold embedded into an Euclidean space, the metric projection yields the particular retraction $R_x^\pi(v) := \pi\pa{x + v}$; see~\cite[Chapter 4]{absilbook}, which is also used in the context of numerically integrating differential equations on embedded submanifolds~\cite{hairer}. 
	\paragraph{Correction} In analogy to the Euclidean case
	from Section~\ref{s.Euclideancontinuation},
	we rely on a second order method for refining the estimate $y_{k+1}$ such that it becomes a (nearly) critical point of $f(\cdot,\lambda_{k+1})$. The tolerance on the Riemannian gradient norm is chosen small enough to closely track the solution curve, typically $10^{-6}$.
	The Riemannian Newton (RN) method~\cite[chapter 6]{absilbook} can be used for this purpose; its basic form is described in Algorithm~\ref{alg.riemannianNewton}. Note that the Riemannian Newton method can be replaced by any locally superlinearly convergent method, e.g., the Riemannian Trust Region (RTR) method \cite{absilbook}[AMS08, Chapter 7] or the Riemannian BFGS method \cite{bfgs}. These methods can take full advantage of sufficiently accurate initial guess provided by the prediction, yielding a fast correction phase. Although a first order method such as steepest descent could in principle be used, they would not benefit the warmstarting fully as they do not exhibit accelerated convergence near a critical point. \\
	\begin{algorithm}
		\caption{\aheader{x^*}{RiemannianNewton}{x^{(0)},f,\mathrm{tol}, N_{\text{inner}}}}
		\label{alg.riemannianNewton}
		\begin{algorithmic}[1]
			\While{$\|\grad f({x^{(j)}})\|> \mathrm{tol}$ $\wedge$ $j \leq N_{\text{inner}}$}
			\State Solve $\hessx{f}(x^{(j)})\pac{n^{(j+1)}} = -\grad{f}(x^{(j)})$;
			\State $x^{(j+1)} = R_{x^{(j)}}(n^{(j+1)})$;
			\EndWhile\\
			\textbf{end}
			\State \textbf{return} $x^* = x^{(j)}$;
		\end{algorithmic}
	\end{algorithm}
	
	\paragraph{Riemannian-Newton Continuation (RNC)}
	The whole predictor-corrector scheme for Riemannian manifolds is sketched in Algorithm~\ref{alg.riemannianContinuation}. The optional adaptive step size strategy in line~\ref{line:adaptive} will be explained in Section~\ref{ss.stepSizeAdaptivity} below.
	
	\begin{algorithm}
		\caption{$\paa{x_k,\lambda_k}$ = RiemannianNewtonContinuation$\pa{x_0, f, N_{\text{steps}}, \mathrm{tol}, N_{\text{inner}}}$}
		\label{alg.riemannianContinuation}
		\begin{algorithmic}[1]
			\State $h_0 = \frac{1}{N_{\text{steps}}}$, $\lambda_0 = 0$, $k = 0$;
			\While{$\lambda_k<1$}
			\If{\text{tangentialPrediction}}
			\State Solve $ \hessx f(x_k,\lambda_k)\pac{t_{k}} = - \ddp{\grad{f}(x_k,\lambda_k)}{\lambda}  $;
			\If{\text{adaptStepSize}}
			\State \label{line:adaptive} Determine the new step size $h_{k}$ with Algorithm~\ref{alg.stepSizeAdaptivity}.
			\EndIf
			\State $y_{k+1} = R_{x_{k}}(h_k t_{k})$;
			\Else
			\State $y_{k+1} =x_{k}$;
			\EndIf
			\State $\lambda_{k+1} =  \min\paa{\lambda_{k}+h_k,1}$
			\State $x_{k+1} = \operatorname{RiemannianNewton}\pa{y_{k+1}, f(\cdot,\lambda_{k+1}),\mathrm{tol}, N_{\text{inner}}}$;
			\If{$ \|\grad f(x_{k+1},\lambda_{k+1})\|>\operatorname{tol}$}
			\State Error("Traversing failed at step k.");
			\Else 
			\State $k = k+1$;
			\EndIf
			\EndWhile\\
			\textbf{end}
			\State \textbf{return} $\paa{x_j,\lambda_j}_{j = 1,\dots,k}$
		\end{algorithmic}
	\end{algorithm}
	
	\subsection{Prediction order analysis} \label{ss.predictionOrder}
	
	An accurate prediction step leads to fast convergence in the correction step (Algorithm~\ref{alg.riemannianNewton}). The concept of order is used in the Euclidean case~\cite[p.238-239]{deuflhard} to qualitatively capture this accuracy. The following definition extends this concept to the  Riemannian case by considering the prediction path $y(h) \in \mcal$, $h>0$, obtained from the prediction step by varying the step size $h$.
	\begin{definition}[Prediction order]
		Let $x(\lambda)$ be the solution curve defined by~\eqref{eq.riemannianDavidenko} for ${\lambda\in\pac{0,1}}$. A prediction path $y(h)$ such that $y(0) = x(\lambda)$ is said to be of \emph{order p} if there exists a constant $\eta_p>0$, such that 
		\begin{equation*}
			d(x(\lambda+h),y(h))\leq \eta_p h^p,\quad\forall\lambda\in[0,1),
		\end{equation*}
		holds for all sufficiently small $h > 0$.
	\end{definition}
	
	In the following we will prove that the prediction orders for the Riemannian classical and tangential prediction schemes match the ones in the Euclidean case. More specifically, the following lemmas show that classical prediction~\eqref{eq.classicalPred} has order 1 while  tangential prediction~\eqref{eq.tangentialPred} has order 2.
	\begin{lemma}
		The classical prediction path $y_c(h) = x(\lambda)$ has order 1.
	\end{lemma}
	\begin{proof}Applying the definition of distance function, we obtain for sufficiently small $h>0$ that
		\begin{align*}
			d(x(\lambda + h),y_c(h)) &= d\pa{x(\lambda+h),x(\lambda)}\leq \int_{\lambda}^{\lambda+h}\|x'(\tau)\|\dd\tau\leq h \underset{\tau\in\pac{\lambda,\lambda+h}}{\max}\paa{\|x'(\tau)\|}\\
			&\leq h\underset{\tau\in\pac{0,1}}{\max}\paa{\|x'(\tau)\|}
		\end{align*}
	\end{proof}
	\begin{lemma}
		If $x(\cdot)\in C^2([0,1))$, the tangential prediction path  $$y_t(h) = R_{x(\lambda)}(ht(x(\lambda),\lambda))$$ has order 2.
	\end{lemma}
	\begin{proof}
		We choose $h$ sufficiently small such that $\lambda + h < 1$, and $x(\lambda)$, $x(\lambda+h)$, $R_{x(\lambda)}(h x'(\lambda))$ lie in the same open neighborhood $\mathcal{U}\in\mcal$, corresponding to the local chart $\varphi$.
		We denote the coordinate representations of $x(\lambda + h)$ and $R_{x(\lambda)}(h x'(\lambda))$ by
		\begin{equation*}
			\hat x(h) = \varphi(x(\lambda + h)), \quad \quad \hat r(h) = \varphi(R_{x(\lambda)}(hx'(\lambda))).
		\end{equation*}
		By the smoothness assumptions on  $x$, $\varphi$, and $R$, it follows that $\hat r$ and $\hat x$ are both two times continuously differentiable. This allows us to write their second order Taylor expansion with Lagrange remainder as :
		\begin{align*}
			\hat x(h) &= \hat x(0) + h\hat x'(0) + \frac{h^2}{2}\hat x''(h_x),\\
			\hat r(h) &= \hat r(0) + h\hat r'(0) + \frac{h^2}{2}\hat r''(h_r),
		\end{align*}
		for some $h_x,h_r\in\pa{0,h}$. By the retraction definition, note that $\hat x(0) = \hat r(0) = \varphi(x(\lambda))$ and using the local rigidity property
		\begin{equation*}
			\hat x'(0) = \hat r'(0) = \D\varphi(x(\lambda))\pac{x'(\lambda)}.
		\end{equation*}
		We now define the line 
		\begin{align*}
			\hat e(\tau) &= (1-\tau)\cdot\hat x(h) + \tau\cdot\hat r(h)\\
			&=  \hat x(0) + h\hat x'(0) + \tau \cdot \frac{h^2}{2}\pa{\hat r''(h_r)-\hat x''(h_x)}.
		\end{align*}
		Because $\varphi\pa{\mathcal{U}}$ is open, this line is contained in  $\varphi\pa{\mathcal{U}}$ for sufficiently small $h$.
		This allows us to define   
		\begin{equation*}
			e(\tau) = \varphi^{-1}(\hat e(\tau)),
		\end{equation*}
		which is a smooth curve on $\mcal$ joining  $x(\lambda + h)$ and $R_{x(\lambda)}(hx'(\lambda))$: 
		\begin{equation*}
			e(0) = x(\lambda + h)\quad\quad e(1) = R_{x(\lambda)}(hx'(\lambda)).
		\end{equation*}
		Taking the derivative with respect to $\tau$ we have that 
		\begin{equation*}
			e'(\tau) = \D\varphi^{-1}(\hat e(\tau))\pac{\hat e'(\tau)} = h^2\cdot \D\varphi^{-1}(\hat e(\tau))\pac{\tfrac{1}{2}\pa{\hat r''(h_r)-\hat x''(h_x)}}.
		\end{equation*}
		This concludes the proof by noting that 
		\begin{align*}
			d(R_{x(\lambda)}(hx'(\lambda)),x(\lambda + h))\! &\leq\! \int_{0}^{1}\!\!\!\|e'(\tau)\|\dd\tau \\
			\!	&\leq\! h^2\! \int_{0}^{1}\!\!\!\|\D\varphi^{-1}(\hat e(\tau))\pac{\tfrac{1}{2}\pa{\hat r''(h_r)-\hat x''(h_x)}}\!\|\,\dd \tau = O(h^2).
		\end{align*}
	\end{proof}
	
	\subsection{Step size adaptivity via asymptotic expansion}\label{ss.stepSizeAdaptivity}
	
	The selection of the step size $h_k$ in Algorithm~\ref{alg.riemannianContinuation} is of crucial importance for its efficiency. A good step size selection should find a balance between the two conflicting goals of attaining fast convergence in each correction step and maintaining a low number of correction steps. 
	
	An overview of existing strategies for the Euclidean case can be found in~\cite{allg, deuflhard}. In the following, we focus on the case of tangential prediction. We propose to generalize to the Riemannian setting a step size selection scheme which is summarized in~\cite[section 6.1]{allg}. It aims at guaranteeing the three following conditions: (i) the distance between the prediction point $y_{k+1}$ and the corresponding solution point $x_{k+1}$ is below a prescribed tolerance, (ii) the RN method on $f(\cdot, \lambda_k + h_k)$ started at the prediction point $y_{k+1}$ is sufficiently contractive and (iii) the curvature of the solution curve between $x_k$ and $x_{k+1}$ is below a prescribed tolerance. For the Euclidean case, an analogous approach intended to fulfill condition (ii) is used in the numerical continuation software package HOMPACK~\cite{hompackLibrary}, while the strategy we now describe targets the three above conditions simultaneously. 
	
	Given any $(w,\lambda)\in\mcal\times\pac{0,1}$ such that $\hessx{f}(w,\lambda)$ is full rank, we denote 
	\begin{itemize}
		\item $t(w,\lambda) = -\hessx{f}(w,\lambda)^{-1}\pac{\ddp{\grad{f}(w,\lambda)}{\lambda}}$ : the prediction vector,
		\item $n(w,\lambda) = -\hessx{f}(w,\lambda)^{-1}\pac{\grad{f}(w,\lambda)}$ : the RN update vector.
	\end{itemize}
	Given $(x(\lambda),\lambda)$ on the solution curve, recall the tangential prediction point as a function of step size $h > 0$ is 
	\begin{equation}\label{eq.predictionPath}
		y( h) = R_{x(\lambda)}( h t(x(\lambda),\lambda)).
	\end{equation}
	An approximation of the distance between $y(h)$ and $x(\lambda + h)$ can be obtained from the norm of the first RN update vector. We shall denote it
	\begin{equation}\label{eq.delta}
		\delta(x(\lambda), \lambda, h) := \|n(y(h),\lambda + h)\|.
	\end{equation} 
	If we let $z(h) = R_{y(h)}\pa{n(y(h), \lambda + h)}$ indicate the first iterate of the RN method, the first contraction rate of the RN is defined as
	\begin{equation}\label{eq.kappa}
		\kappa(x(\lambda), \lambda, h) := \frac{\|n(z(h),\lambda + h)\|}{\|n(y(h),\lambda + h)\|}.
	\end{equation}
	Upon convergence of the RN method for $f(\cdot, \lambda + h)$, this ratio is smaller than 1. Finally, the curvature of the solution curve between two points $x(\lambda)$ and $x(\lambda + h)$ can be approximated with
	\begin{equation}\label{eq.alpha}
		\alpha(x(\lambda), \lambda, h) := \acos\pa{\scalp{\frac{t(x(\lambda),\lambda)}{\| t(x(\lambda),\lambda)\|}}{\frac{\tcal_{y(h)\to x(\lambda)}( t(y(h),\lambda+ h))}{\|\tcal_{y(h)\to x(\lambda)}( t(y(h),\lambda+ h))\|}}_{x(\lambda)}},
	\end{equation}
	the angle between the prediction vector at the solution curve point $(x(\lambda),\lambda))$ and the prediction vector at the prediction point $(y(h),\lambda + h)$. In order to measure their relative angle we transport ${t(y(h),\lambda+ h)\in T_{y(h)}\mcal}$ to ${T_{x(\lambda)}\mcal}$ using a linear map ${\tcal_{y(h)\to x} : T_{y(h)}\mcal\to T_{x(h)}\mcal}$ which can be either parallel transport along the prediction curve $y(h)$ or, more generally, a transporter~\cite[Definition 10.61]{boumalBook}. Note that~\eqref{eq.alpha} is well defined only if $t(x,\lambda)\ne 0$, which also guarantees $t(y(h),\lambda+ h))$ is non zero for sufficiently small $h$.
	
	The following lemma inspired by~\cite[Lemmas 6.1.2, 6.1.8]{allg} is the cornerstone of the step selection strategy. It provides a Taylor expansion with respect to $h$ around $h=0$ of the indicators~\eqref{eq.delta},~\eqref{eq.kappa},~\eqref{eq.alpha}.
	\begin{lemma}\label{lem.asymptExpansion}
		Assume $f\in C^4$. If for each $(x,\lambda)$ of the solution curve we have
		\begin{equation}\label{eq.nonDegeneracy}
			\Dndtz{n(y(h),\lambda+h)}{h}{2}\ne 0,
		\end{equation}
		where $\Dndt{}{h}{2}$ denote the second covariant derivative along the prediction path~\eqref{eq.predictionPath}. Then there exist functions $\delta_2(x, \lambda)$, $\kappa_2(x, \lambda)$,  $\alpha_1(x, \lambda)$ only depending on $x$ and $\lambda$ such that the following holds:
		\begin{enumerate}[label=(\roman*)]
			\item The norm of the first Newton update vector $\delta(x, \lambda, h) = \|n(y(h),\lambda + h)\|$ verifies
			\begin{equation*}
				\delta(x, \lambda, h) = \delta_2(x, \lambda)h^2 + \Ocal{h^3}.
			\end{equation*} 
			\item Newton's method contraction rate $\kappa(x,h) = \dfrac{\|n(z(h),\lambda + h)\|}{\|n(y(h),\lambda + h)\|}$ verifies 
			\begin{equation*}
				\kappa(x, \lambda, h) = \kappa_2(x, \lambda)h^2 + \ocal{h^2}.
			\end{equation*} 
			\item If $t(x,\lambda) \ne 0$, the prediction angle $$\alpha(x,h) = \acos\pa{\scalp{\frac{t(x,\lambda)}{\| t(x,\lambda)\|}}{\frac{\tcal_{y(h)\to x}( t(y(h),\lambda+ h))}{\|\tcal_{y(h)\to x}( t(y(h),\lambda+ h))\|}}_{x(\lambda)}}$$ is well defined, and provided that
			\begin{equation}\label{eq:missingAssumption}
				\Ddtz{\tcal_{y(h)\to x}( t(y(h),\lambda+ h))}{h} \ne c t(x,\lambda),\quad\forall\,c\in\Rbb,
			\end{equation}			
			it verifies
			\begin{equation*}
				\alpha(x, \lambda, h) = \alpha_1(x, \lambda)h + \Ocal{h^2}.
			\end{equation*} 
			
		\end{enumerate}
	\end{lemma}
	The proof of Lemma~\ref{lem.asymptExpansion} can be found in appendix~\ref{a:proofLemma3}. We now describe the step size selection strategy inspired by this result. Given positive constants $\delta_{\max}$, $\kappa_{\max}$ and $\alpha_{\max}$, we aim at finding the largest $h_k>0$ such that 
	\[
	\delta(x_k,\lambda_k, h_k) \leq \delta_{\max},\quad
	\kappa(x_k,\lambda_k, h_k) \leq \kappa_{\max},\quad 
	\alpha(x_k,\lambda_k, h_k) \leq \alpha_{\max}.
	\]
	Given a trial step size $\tilde h_k$ (obtained, e.g., from the previous step), Lemma~\ref{lem.asymptExpansion} allows us to estimate 
	\begin{align}
		\label{eq.deltaTilde}\delta_2(x_k,\lambda_k) \simeq\tilde \delta_2(x_k, \lambda_k) := \sqrt{\frac{\delta(x_k, \lambda_k,\tilde h_k)}{\tilde h_k^2}},\\
		\label{eq.kappaTilde}\kappa_2(x_k,\lambda_k) \simeq\tilde \kappa_2(x_k, \lambda_k) = \sqrt{\frac{\kappa(x_k, \lambda_k,\tilde h_k)}{\tilde h_k^2}}.\\
		\label{eq.alphaTilde}\alpha_1(x_k,\lambda_k) \simeq\tilde \alpha_1(x_k, \lambda_k) = \frac{\alpha(x_k, \lambda_k,\tilde h_k)}{\tilde h_k},
	\end{align} 
	Then, imposing
	\begin{align*}
		\tilde \delta_2(x_k, \lambda_k) h_k^2 \leq \delta_{\max},\quad 
		\tilde \kappa_2(x_k, \lambda_k) h_k^2 \leq \kappa_{\max},\quad
		\tilde \alpha_1(x_k, \lambda_k) h_k \leq \alpha_{\max},
	\end{align*}
	yields 
	\begin{equation*}
		h_k \leq \tilde h_k \min\paa{ \sqrt{\frac{\delta_{\max}}{\tilde\delta(x_k,\lambda_k)}}, \sqrt{\frac{\kappa_{\max}}{\tilde\kappa(x_k,\lambda_k)}}, \frac{\alpha_{\max}}{\tilde\alpha(x_k,\lambda_k)}}.
	\end{equation*}
	This is the criterion to adjust step size, but not to make too drastic changes in the step size, the increase is limited to a factor of 2 and the decrease to a factor $\frac{1}{2}$. The resulting procedure is summarized in Algorithm~\ref{alg.stepSizeAdaptivity}. Note that this comes at the non-negligible cost of (approximately) solving 3 extra linear systems involving the Riemannian Hessian.
	\begin{algorithm}
		\caption{\aheader{ h_{k}}{AdaptiveStepSize}{\tilde h_k,x_k,\lambda_k,t(x_k,\lambda_k),f, \alpha_{\max},\delta_{\max},\kappa_{\max}}}
		\label{alg.stepSizeAdaptivity}
		\begin{algorithmic}[1]
			\State $y_k = R_{x_{k}}(\tilde h_k t(x_k,\lambda_k))$;
			\State Solve $ \hessx f(y_k,\lambda_k+\tilde h_k)\pac{t(y_k,\lambda + \tilde h_k)} = - \ddp{\grad f(y_k,\lambda_k + \tilde h_k)}{\lambda}  $;
			\State Solve $\hessx{f}(y_k,\lambda_{k} + \tilde h_k)\pac{n(y_k,\lambda + \tilde h_k)} = -\grad{f}(y_{k},\lambda_k + \tilde h_k)$;
			\State $z_k = R_{y_k}(n(y_k,\lambda + \tilde h_k))$;
			\State Solve $\hessx{f}(z_k,\lambda_{k} + \tilde h_k)\pac{n(z_k, \lambda + \tilde h_k)} = -\grad{f}(z_{k},\lambda_k + \tilde h_k)$;
			\State Compute $\tilde \delta_2$, $\tilde \kappa_2$ and $\tilde \alpha_1$ using~\eqref{eq.deltaTilde},~\eqref{eq.kappaTilde} and~\eqref{eq.alphaTilde}. 
			\State $h_{k} = \tilde h_k\max\paa{\frac{1}{2},\min\paa{ \sqrt{\dfrac{\delta_{\max}}{\tilde\delta_2}}, \sqrt{\dfrac{\kappa_{\max}}{\tilde\kappa_2}}, \dfrac{\alpha_{\max}}{\tilde\alpha_1}, 2}}$;
			\State \textbf{return} $h_{k}$
		\end{algorithmic}
	\end{algorithm}
	\section{Application to the Karcher mean of symmetric positive definite matrices}\label{s.karcherMean}
	
	In this section, we apply RNC, Algorithm~\ref{alg.riemannianContinuation}, to a classical problem of Riemannian optimization: the computation of the  Karcher mean, also referred to as Riemannian center of mass~\cite{karchermeanfirstpaper}.
	Given $K$ points $p_1,\dots,p_K\in\mcal$ the Karcher mean (with uniform weights) is defined as 
	\begin{equation} \label{eq:karcher}
		\underset{q\in\mcal}{\arg\min}\paa{\sum_{i=1}^{K}d(q,p_i)^2},
	\end{equation}
	where $d(q,p_i)$ is the distance function on $\mcal$.
	This optimization problem admits a unique solution for any manifold provided that all $p_i$ are sufficiently close to each other. This requirement can be dropped for instance in the case of complete Riemannian manifolds of non-positive sectional curvature, also called Cartan-Hadamard manifolds, for which the Karcher mean is always uniquely defined for any set of points~\cite{karcherMeanUniqueness}.
	
	We will focus on the Karcher mean of $n\times n$ real symmetric positive definite matrices
	\begin{equation*}
		\SPD{n} = \paa{A\in \Rbb^{n\times n}\,:\,A = A^T,\, v^TAv>0\,\forall v\in\Rbb^n\text{ if } v\ne0}.
	\end{equation*} 
	In the following, we recall facts from~\cite{bhatia} on a suitably chosen Riemannian manifold structure of $\SPD{n}$.
	
	Clearly, $\SPD{n}$ is an open cone of the vector space of symmetric matrices $\SYM{n}$. The tangent space at $A\in\SPD{n}$ can be identified with this vector space:
	\begin{equation*}
		T_A\SPD{n} \simeq \SYM{n}.
	\end{equation*} 
	The Thompson or statistical metric
	makes $\SPD{n}$ a Cartan-Hadamard manifold; it has the following expression
	\begin{equation}
		\scalp{V}{W}_A = \trace\pa{A^{-1}VA^{-1}W},\quad \forall \,V,W \in T_A\SPD{n}\simeq\SYM{n}.
	\end{equation} 
	With this metric, a geodesic joining ${A,B\in \SPD{n}}$ is given by
	\begin{equation}\label{eq.SPDgeodesic}
		\gamma_{AB}(t) = A\exp(t\log(A^{-1}B)),
	\end{equation}
	where $\exp$ and $\log$ are the matrix exponential and logarithm. In turn, the distance function reads as
	$
	d(A,B) = \|\log(A^{-\frac{1}{2}}BA^{-\frac{1}{2}})\|_F
	$
	and the Karcher mean problem~\eqref{eq:karcher} becomes
	\[
	\underset{X\in \SPD{n}}{\arg\min} f(X), \qquad f(X):= \sum_{i = 1}^{K}\|\log(A_i^{-\frac{1}{2}}XA_i^{-\frac{1}{2}})\|_F^2,
	\]
	with $A_1,\dots,A_K \in \SPD{n}$.
	
	An expression for the Riemannian gradient $f$ and for the Riemannian Hessian of $f$ associated to the Levi-Civita connection compatible with the Thompson metric can be found in~\cite[Equations 4.6 and 4.16]{surveyGeometricMean}. For numerical experiments, we consider the second order retraction~\cite[Equation 4.10]{surveyGeometricMean} and the transporter given by parallel transport along geodesics~\cite[Equation 3.4]{vecTranspSPDPaper}.

	\subsection{Homotopy for the Karcher mean problem}
	The Riemannian manifold structure for $\SPD{n}$ introduced in the previous section, makes the Karcher mean of positive definite matrices a strictly geodesically convex problem. 
	This implies that standard Riemannian optimization algorithms can successfully solve the problem without the need of numerical continuation. Nevertheless, we use this application as a model problem for the purpose of testing the RNC algorithm and illustrating its behavior.
	
	We propose the following homotopy for the Karcher mean of $A_1,\dots,A_K \in \SPD{n}$.
	We define $K$ smooth curves $B_i\::\:\pac{0,1}\to\SPD{n}$ such that 
	\begin{equation*}
		B_i(1) = A_i,\quad\forall i = 1, \dots, K,
	\end{equation*}
	and such that the Karcher mean of $B_1(0),\dots,B_K(0)$ can be solved easily. In particular, this is the case when all starting points are equal, $B_1(0) = \cdots = B_K(0) = A_0$. In our experiments, we choose $A_0 = I_{n \times n}$. For $B_i$, we choose the geodesic curve~\eqref{eq.SPDgeodesic} joining $A_0$ to $A_i$, that is, 
	\begin{equation*}
		B_i(\lambda) = A_0 \exp(\lambda \log(A_0^{-1}A_i)).
	\end{equation*}
	We can now write the parametric Karcher mean optimization problem as 
	\begin{equation}\label{eq.KarcherMeanHomotopy}
		\underset{X\in\SPD{n}}{\arg\min}\paa{f(X,\lambda) = \sum_{i = 1}^{K}\|\log(B_i(\lambda)^{-\frac{1}{2}}XB_i(\lambda)^{-\frac{1}{2}})\|_F^2},\quad \forall\lambda\in\pac{0,1}.
	\end{equation}

	Using the parameter dependent expression of the Riemannian gradient of \eqref{eq.KarcherMeanHomotopy}, its derivative with respect to the parameter $\lambda$, needed for performing tangential continuation, is given by 
	\begin{equation*}
		\ddp{\grad f(X,\lambda)}{\lambda} = - 2 \sum_{i = 1}^K X\D\log(X^{-1}B_i(\lambda))\pac{X^{-1}B_i'(\lambda)},
	\end{equation*}
	where $B_i^\prime(\lambda) = A_0 \exp(\lambda \log(A_0^{-1}A_i))\log(A_0^{-1}A_i)$ and  $\D\hspace{0.02cm}\log(X)[\cdot]$ is the Fréchet derivative of the matrix logarithm; see~\cite{dlog} for its computation.

	\subsection{Numerical results}
	
	All numerical experiments presented in this paper have been performed in Matlab 2019b, using the Matlab Riemannian optimization library Manopt~\cite{manoptPaper}.
	
	In all experiments, we consider computing the Karcher mean for a set of $K = 75$ symmetric positive definite matrices of size $n = 10$ that are built from their eigenvalue decomposition 
	\begin{equation*}
		A_i = V_i D_i V_i^T,\quad \forall i = 1,\dots, K,
	\end{equation*}
	where $V_i$ is a random orthogonal matrix and $D_i$ a diagonal matrix. For the diagonal entries, 9 are chosen at random in the interval $\pac{1,2}$ and the last one is chosen such that the matrices have a large but still moderate condition number (approximately $10^3$). Figure~\ref{fig.KarcherMeanEasyInstance} compares the direct optimization with the standard RN method and the continuation approach (tangential RNC with fixed step size $N_{\text{steps}} = 3$) using the homotopy~\eqref{eq.KarcherMeanHomotopy}. For all experiments, we used the  identity matrix as initial condition, $\operatorname{tol} = 10^{-6}$ and $\operatorname{N_{\text{inner}}} = 5000$. Note that other choices, like the planar approximations of the Karcher mean discussed in~\cite{surveyGeometricMean}, are possible. For this example, it turns out that the RN method enters 
	a superlinear convergence regime from the beginning (as seen from the concavity of the black convergence curve) and thus solves the problem in very few iterations. For such a simple instance, the continuation approach does not offer advantages.
	\begin{figure}
		\centering
		\includegraphics[trim=4.1cm 0cm 4.1cm 0,clip,width=0.85\textwidth]{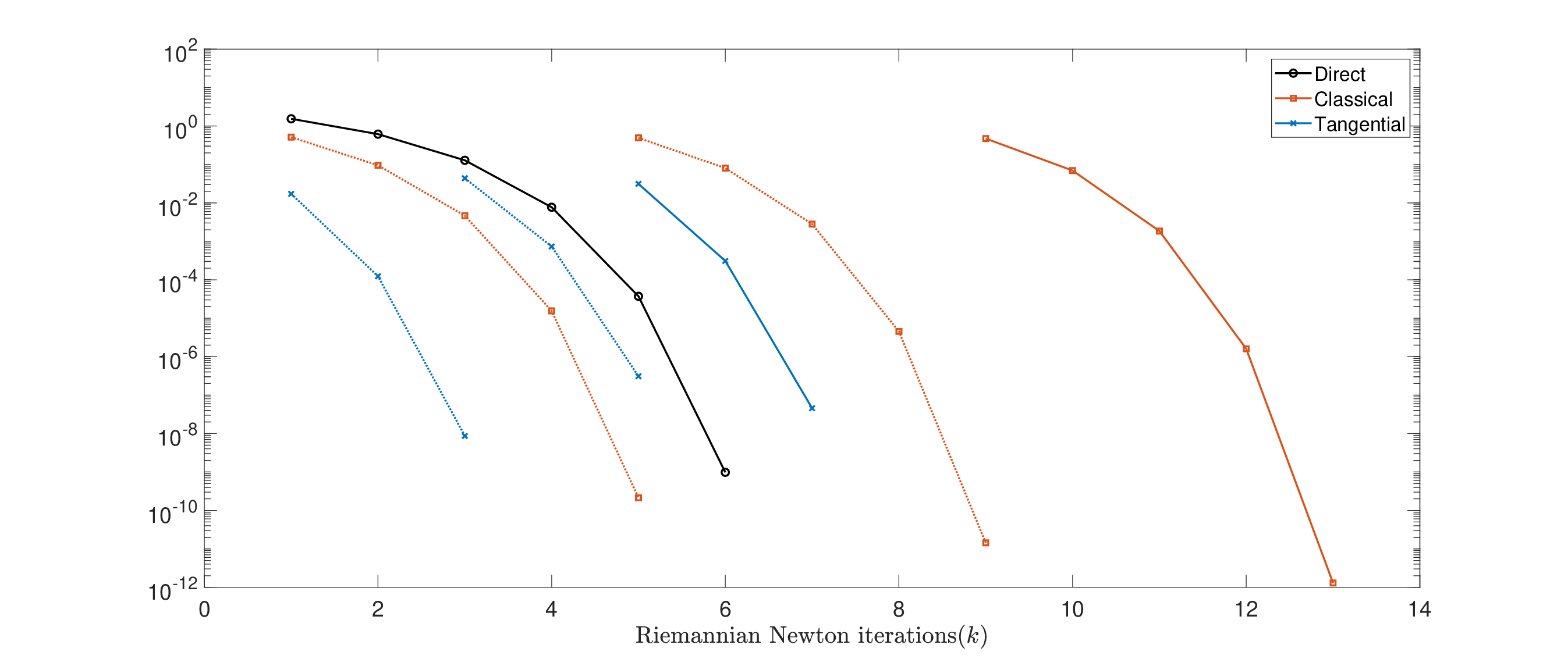}
		\caption{\footnotesize
			Convergence of the Riemannian gradient norm versus RN iterations for a non-pathological instance of the Karcher mean problem. The iterations needed by the (plain) RN method is compared to the total number of RN correction steps needed by fixed step size classical and tangential prediction RNC ($N_{\text{steps}} = 3$). The Riemannian gradient norm for $\lambda = 1$ is plotted with solid lines, whereas we use dashed lines for intermediate values of $\lambda$. }
		\label{fig.KarcherMeanEasyInstance}
	\end{figure}
	
	In order to better highlight the advantage of the RNC algorithm, we choose a somewhat pathological instance: the diagonal matrices $D_i$ are chosen such that their condition number is $10^8$. Half of the diagonal entries are exponentially distributed in $\pac{0.1, 1}$ and the other half exponentially distributed in $\pac{10^6, 10^7}$. In turn, the optimization problem is highly ill-conditioned, leading to stagnation in the initial phase of the RN method; see Figure~\ref{fig.KarcherMeanFixedStepSize}. In contrast, the RNC algorithm~\ref{alg.riemannianContinuation} with fixed number of steps $N_{\text{steps}} = 2$ does not suffer from such stagnation during the correction phase. In turn, the total number of RN iterations is reduced. Tangential prediction leads to slightly better convergence compared to classical prediction, but it also comes at the cost of solving an extra linear system, which leads to a less favorable computational time; see Table~\ref{tab.summaryKarcherMeanExperiments}. The number of fixed steps in Figure~\ref{fig.KarcherMeanFixedStepSize} is chosen to best highlight the slight improvement of RNC over direct RN optimization. However, for this particular application the advantage disappears when an automatic step sizing strategy is used. Nevertheless, the step size adaptivity results for different set of hyperparameters $(\kappa_{\max}, \alpha_{\max}, \delta_{\max})$ in Table~\ref{tab.summaryKarcherMeanExperiments} illustrate the need for a compromise to be found between the number of corrections and the length of each correction. This is further demonstrated by Figure~\ref{fig.KarcherMeanVaryingNSteps} where the computational effort for fixed step size RNC is reported for different values of $N_{\mathrm{steps}}$. 
	\begin{figure}
		\centering
		\includegraphics[width=0.85\textwidth]{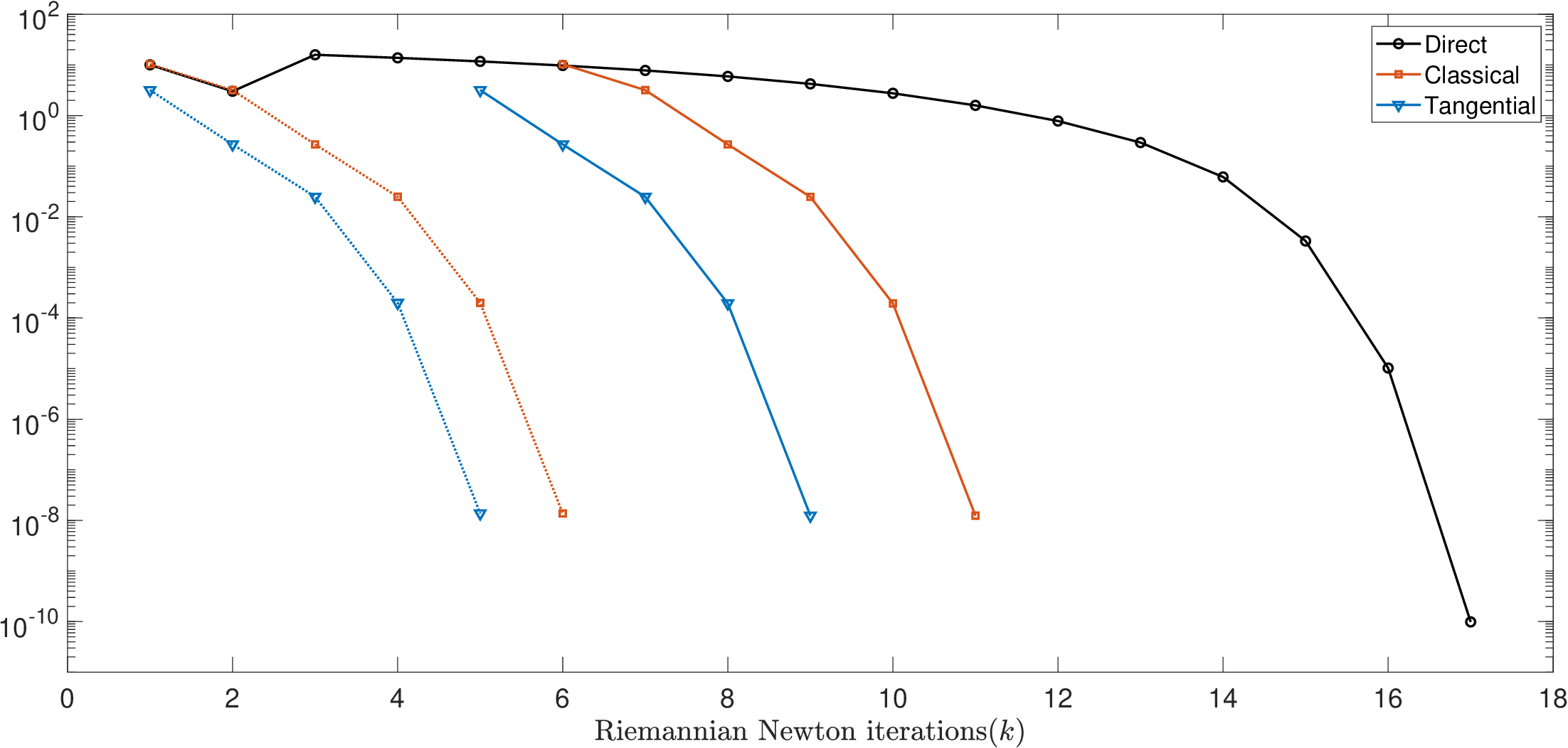}
		\caption{\footnotesize 
			Convergence of the Riemannian gradient norm (of the original problem in solid lines and of each intermediate problem in dashed lines) versus RN iterations for the pathological instance of the Karcher mean problem. RN method is compared with fixed step size classical and tangential prediction RNC algorithm ($N_{\text{steps}} = 2$).}
		\label{fig.KarcherMeanFixedStepSize}
	\end{figure}
	\begin{figure}\label{fig.KarcherMeanVaryingNSteps}
		\centering
		\begin{minipage}{0.495\textwidth}
			\includegraphics[trim=0.75cm 0cm 1.75cm 0.5cm,clip,width=\textwidth]{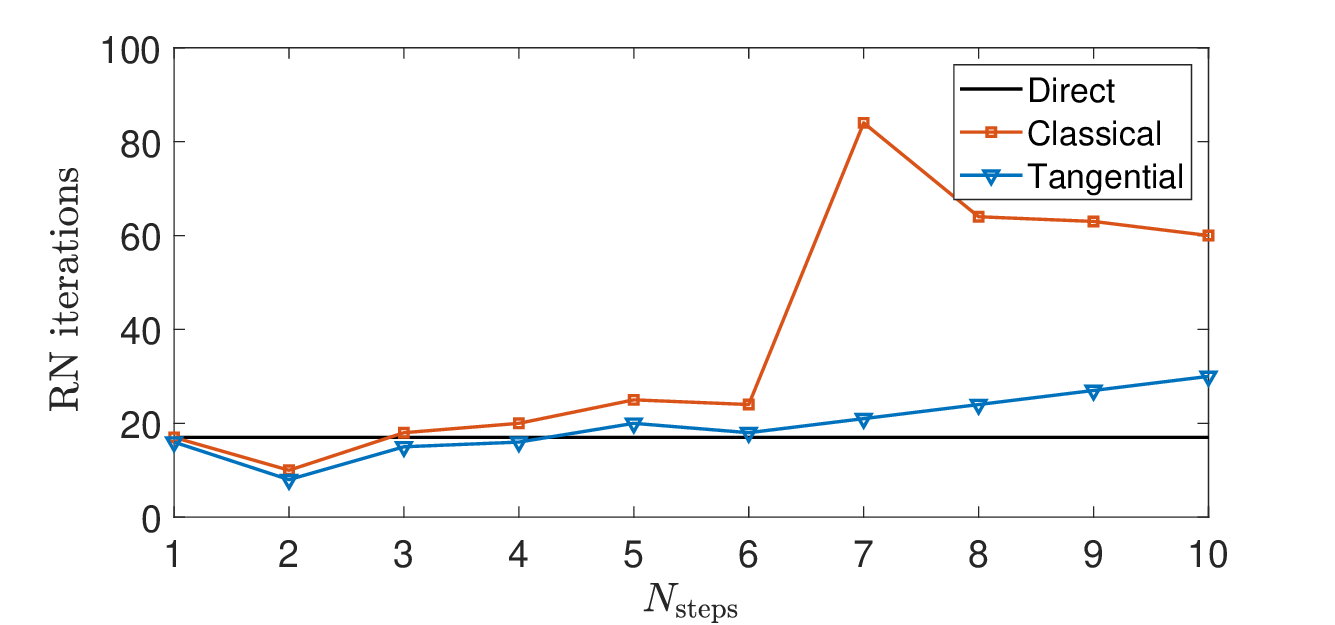}
		\end{minipage}
		\begin{minipage}{0.495\textwidth}
			\includegraphics[trim=0.75cm 0cm 1.75cm 0.5cm,clip,width=\textwidth]{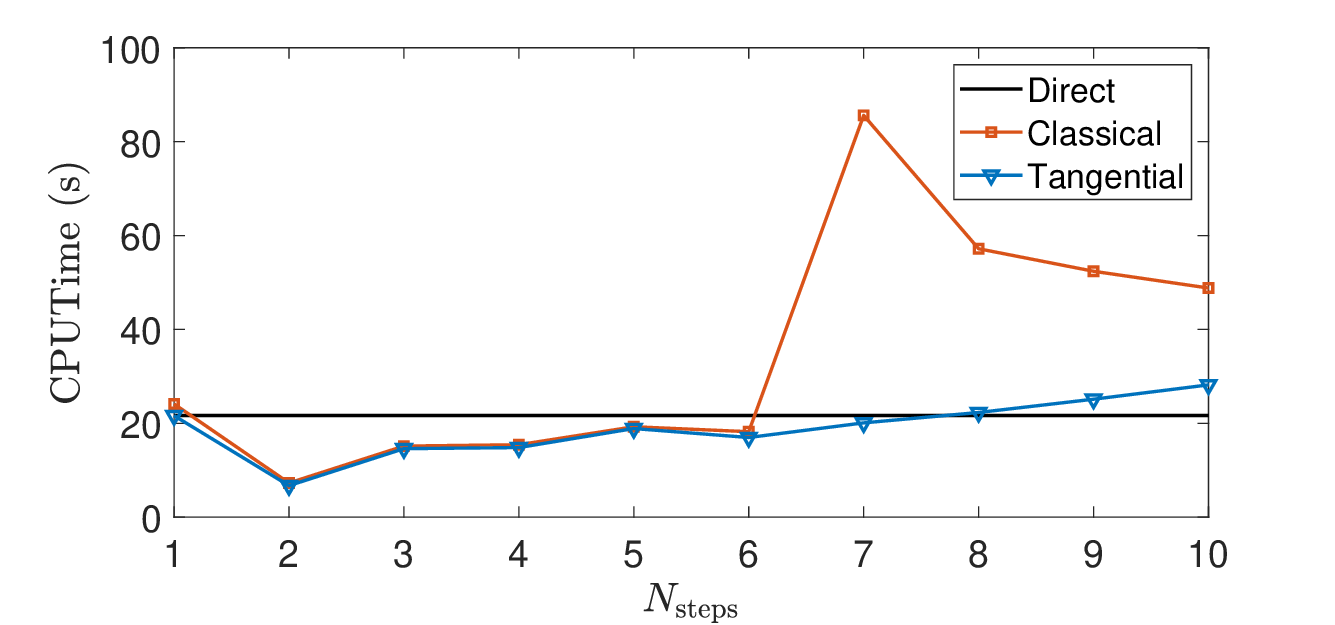}
		\end{minipage}
		\caption{\footnotesize
				RN iterations (left) and computation time (right) versus the number of continuation steps for the fixed step size RNC on the pathological instance of the Karcher mean problem.}
	\end{figure}
	\begin{table}
		\centering
		\caption{\footnotesize Summary of the number of iterations and computation times for the numerical experiments on the Karcher mean pathological instance. The hyperparameters $(\kappa_{\max}, \alpha_{\max}, \delta_{\max})$ for the step size adaptive experiments (1), (2) and (3) are respectively $(0.6, 3^\circ, 10)$, $(0.3, 1.5^\circ, 5)$ and $(0.15, 0.75^\circ, 2.5)$.}
		\begin{tabular}{c|c|c|c|}
			\cline{2-4}
			& \multicolumn{3}{c|}{\textbf{Karcher mean}} \\ \hline
			\multicolumn{1}{|c|}{Direct Optimization (RN)} & 1 & 17 & 20.04 \\ \hline
			\multicolumn{1}{|c|}{Fixed step size classical RNC} & 2 & 11 & 6.65 \\ \hline
			\multicolumn{1}{|c|}{Fixed step size tangential RNC} & 2 & 9 & 6.32 \\ \hline
			\multicolumn{1}{|c|}{Step size adaptive RNC (1)} & 3 & 22 & 45.66 \\ \hline
			\multicolumn{1}{|c|}{Step size adaptive RNC (2)} & 3 & 17 & 32.36 \\ \hline
			\multicolumn{1}{|c|}{Step size adaptive RNC (3)} & 6 & 25 & 56.96 \\ \hline
			& Corrections & Correction iterations & Time (s) \\ \cline{2-4} 
		\end{tabular}
		\label{tab.summaryKarcherMeanExperiments}
	\end{table}
	\section{Application to low-rank matrix completion}\label{s.matrixCompletion}
	In matrix completion, only some entries of a matrix $A \in\Rbb^{m\times n}$ are available and the goal is to determine the rest of the entries. This is clearly an ill-posed problem and one way to regularize it is to impose low-rank constraints; see \cite{matrixCompletionSurvey} for a recent review on existing methods.
	In the following, we describe the Riemannian optimization setting introduced by~\cite{bart}. 
	
	We let ${\Omega \subset \paa{1,\dots,m}\times\paa{1,\dots,n}}$ contain the indices $(i,j)$ for which $A_{ij}$ is known and define the projection 	
	\begin{equation}
		P_\Omega(A)=
		\begin{cases}
			A_{ij} & \text{if }\pa{i,j}\in\Omega\\
			0 & \text{if }\pa{i,j}\notin\Omega.
		\end{cases}
	\end{equation}
	We aim at approximating $A$ by a matrix of a given fixed rank $k\ll \min\{m,n\}$ or, equivalently, by a matrix from the set
	\begin{equation*}
		\mcal_k = \paa{U\Sigma V^T\!\in\!\matr{m}{n}\! : U\!\in\! \operatorname{St}(m,k),V\!\in\!\operatorname{St}(n,k), \Sigma = \operatorname{diag}(\sigma_i), \sigma_1\geq\dots\geq\sigma_k>0},
	\end{equation*}  
	where $\operatorname{St}(m,k) = \paa{U\in \matr{m}{k}:U^TU = I_k}$ is the Stiefel manifold. It can be shown that $\mcal_k$ is a smooth manifold of dimension $k(m+n-k)$. This leads to the following smooth Riemannian optimization formulation:
	\begin{equation}\label{eq:RiemannianCompletionProblem}
		\underset{X\in\mcal_k}{\min} f(X), \quad f(X) := \frac{1}{2}\|P_{\Omega}(X)-A_\Omega\|_F^2,
	\end{equation}
	with $A_\Omega = P_\Omega(A)$.

	The fixed rank manifold $\mcal_k$ is endowed the standard structure of Riemannian submanifold of $\matr{m}{n}$ as presented in \cite[Section 2]{bart}. The expressions for the Riemannian gradient and the Riemannian Hessian are given in \cite[Equation 11 and Proposition 2.2]{bart}. For the numerical experiments, we opted for the metric projection retraction \cite[Equation 13]{bart} and the orthogonal projection to the destination tangent space \cite[Equation 14]{bart} for the transporter.

	\subsection{Homotopy for the matrix completion}\label{ss.homotopyCurveMatrixCompletion}
	
	The homotopy we propose for the matrix completion problem consists of replacing $A_\Omega$ in~\eqref{eq:RiemannianCompletionProblem} with a smooth curve  $B_\Omega(\lambda)\in\matr{m}{n}$, $\lambda\in\pac{0,1}$, such that $B_\Omega(1) = A_\Omega$. If we take $B_\Omega(0) = P_\Omega(A_0)$, for some known matrix $A_0$ of rank $k$, then the first point of the continuation solution curve is $A_0$ itself. If we let $\pi:\matr{m}{n}\to\mcal_k$ denote the rank-$k$ truncated singular value decomposition, we use $A_0 = \pi\pa{\mathcal{F}(A_\Omega)}$, where $\mathcal{F}$ does not alter the known entries of $A_\Omega$ and imputes the unknown entries via a heuristic procedure. For example, it is common to use zeros for the unknown entries when initializing Riemannian optimization applied to~\eqref{eq:RiemannianCompletionProblem} \cite{matrixCompletionSurvey, boumalOnMatrixCompletion}. In our experiments, we found it more effective to replace missing entries by averaging neighboring known values.
	
	The parametric matrix completion problem is given by
	\begin{equation}\label{eq.matrixCompletionHomotopy}
		\underset{X\in\mcal_{k}}{\min}\paa{f(X,\lambda) = \frac{1}{2}\|P_{\Omega}(X)-B_{\Omega}(\lambda)\|_F^2},\quad \forall\lambda\in\pac{0,1},
	\end{equation}
	with 
	\begin{equation}\label{eq.matrixCompletionInstanceCurve}
		B_{\Omega}(\lambda) = (1-\lambda)P_\Omega(\pi(\mathcal{F}(A_{\Omega}))) + \lambda A_\Omega.
	\end{equation}
	
	From the parameter dependent expression of the Riemannian gradient of \eqref{eq.matrixCompletionHomotopy}, the linearity of $P_\Omega$ and of the tangent space projection $\Pi(X):\matr{m}{n}\to T_X\mcal_{k}$, we obtain
	\begin{equation}
		\ddp{\grad{f}(X,\lambda)}{\lambda} = \Pi(X) \pa{A_\Omega - P_\Omega(\pi(\mathcal{F}(A_{\Omega})))}.
	\end{equation}

	\subsection{Numerical results}
	We apply the RNC Algorithm to an instance of the matrix completion problem where the matrix $A$ is obtained by sampling a bivariate smooth function $g$ on a regular grid of ${\pac{a,b}\times\pac{c,d}}$,
	\begin{equation*}
		A_{i,j} = g\pa{a + i\frac{(b-a)}{m-1}, c + j\frac{(d-c)}{n-1}}, \quad \forall\, i = 0,\dots, m-1, \forall \,j = 1,\dots, n-1.
	\end{equation*}
	We then set $A_\Omega = P_\Omega(A)$, with a randomly generated observation operator $P_\Omega$. We choose the number of known entries accordingly with the rank chosen for $\mcal_k$ using the oversampling rate defined as  
	\begin{equation*}
		\operatorname{OS} = \frac{|\Omega|}{\operatorname{dim}(\mcal_k)} = \frac{|\Omega|}{k(m+n-k)},
	\end{equation*} 
	where $|\Omega|$ is the cardinality of $\Omega$. 
	The matrix $A$ is known to exhibit exponentially decaying singular values, which -- as we will see -- deteriorates the convergence of direct Riemannian optimization methods for~\eqref{eq:RiemannianCompletionProblem}. In particular, we consider the function 
	\begin{equation*}
		g(x,y) = e^{-\frac{(x-y)^2}{\sigma}}
	\end{equation*}
	with $\sigma = 0.1$.
	This function is sampled on $\pac{-1,1}^2$ with a regular grid of $m = n = 300$ points in each direction. We choose the rank $k = 15$ and set $\operatorname{OS} = 3$, implying that $29.25\%$ of the entries are observed.
	
	As the standard RN method tends to fail for this kind of problems, we substituted it with the Riemannian Trust Region algorithm (RTR), both as a corrector at line 11 of algorithm \ref{alg.riemannianContinuation} and as a direct optimization scheme.
	
	\begin{figure}
		\centering
		\includegraphics[trim=4cm .25cm 3.5cm 0,clip,width=0.85\textwidth]{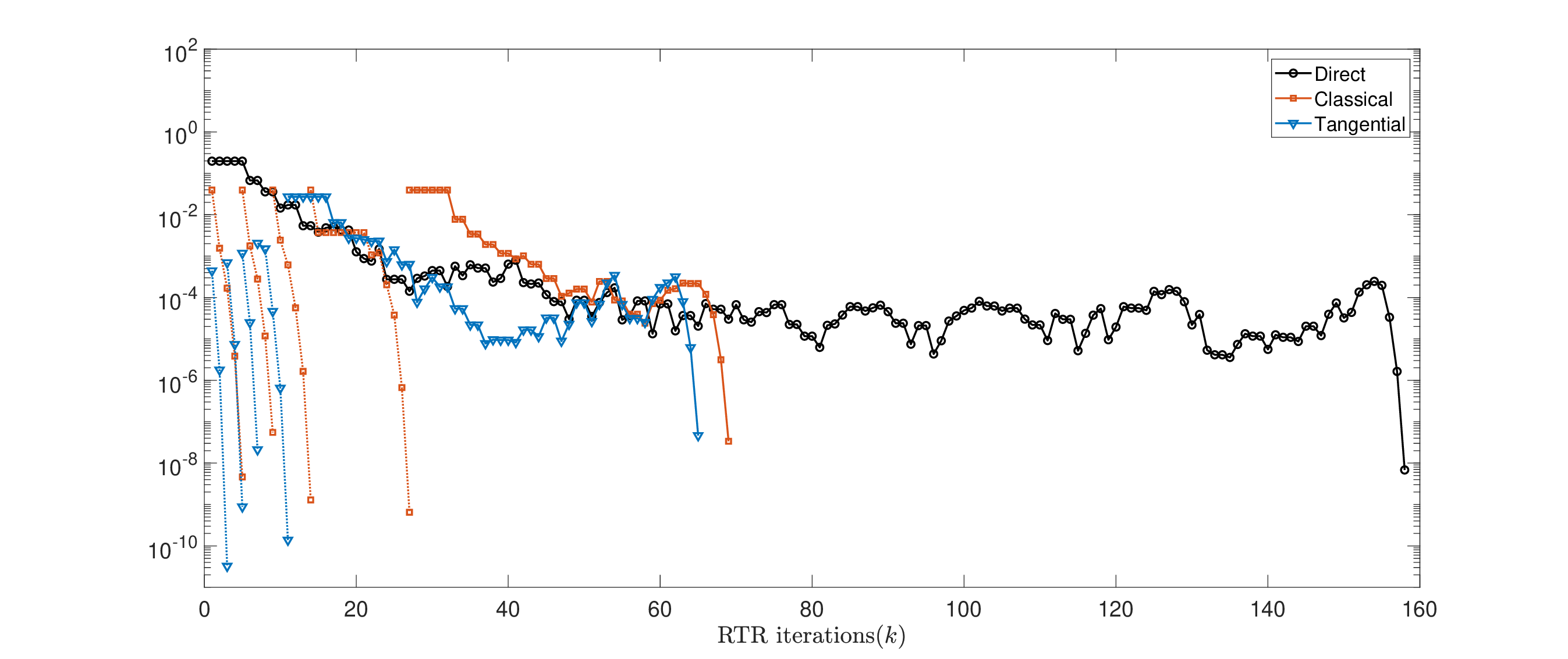}
		\caption{\footnotesize Convergence of the Riemannian gradient norm (of the original problem in solid lines and of each intermediate problem in dashed lines) versus RTR iterations on the matrix completion problem. We compare (plain) RTR optimization initialized at $A_0$ with fixed step size classical and tangential prediction RNC algorithm ($N_{\text{steps}} = 5$) on the matrix completion problem. }
		\label{fig.MatrixCompletionFixedStepSize}
	\end{figure}
	The results of the direct optimization with RTR initialized at $A_0$ compared with fixed step size continuation $N_{\mathrm{steps}} = 5$ on the homotopy using the instance curve~\eqref{eq.matrixCompletionInstanceCurve}  can be seen in Figure~\ref{fig.MatrixCompletionFixedStepSize}. For all experiments we set $\operatorname{tol} = 10^{-7}$ and $N_{\text{inner}} = 5000$. The direct method suffers a long stagnation before entering the superlinear convergence regime. The same stagnation occurs in the last corrections of the continuation procedures, yet less severely and thus the continuation scheme showed to be globally faster both in number of RTR iterations and computation time as summarized in Table~\ref{tab.summaryMatrixCompletionExperiments}. The table also report experiments conducted with two other widely used methods for low-rank matrix completion, namely the Riemannian Conjugate Gradient, referred to as LRGeomCG~\cite{bart}, and the alternating least-squares approach LMAFit~\cite{lmafit}. To make a fair comparison, both use the same initial condition $A_0$ and the stopping criterion is based on the final relative residual on the known entries that the direct RTR method achieves. In Figure~\ref{fig.MatrixCompletionVaryingNSteps}, the best compromise in terms of computation time of fixed step size RNC between the number of continuation steps and the number of steps of each correction is found to be for $N_{\mathrm{steps}} = 3$. If we increase the number of continuation steps, convergence on each correction requires less steps so the total number of RTR does not increase significantly, however the computation time increases due to the fixed costs of each correction.
	\begin{figure}\label{fig.MatrixCompletionVaryingNSteps}
		\centering
		\begin{minipage}{0.495\textwidth}
			\includegraphics[trim=0.75cm 0cm 1.75cm 0.5cm,clip,width=\textwidth]{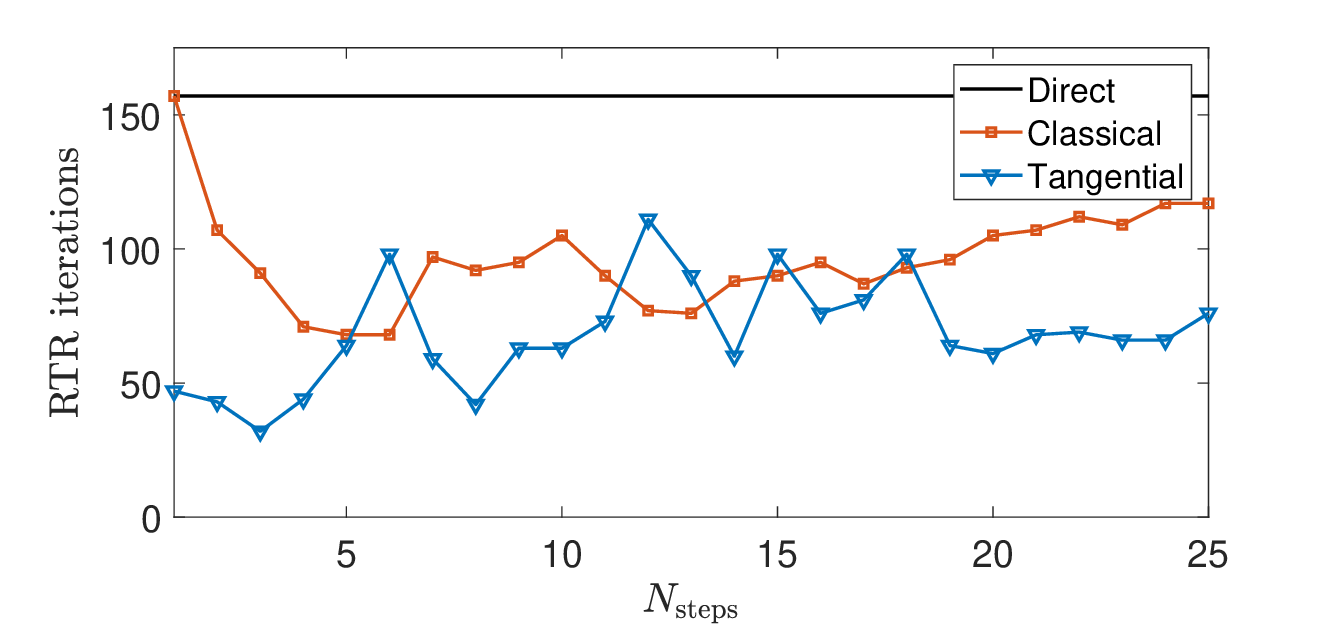}
		\end{minipage}
		\begin{minipage}{0.495\textwidth}
			\includegraphics[trim=0.75cm 0cm 1.75cm 0.5cm,clip,width=\textwidth]{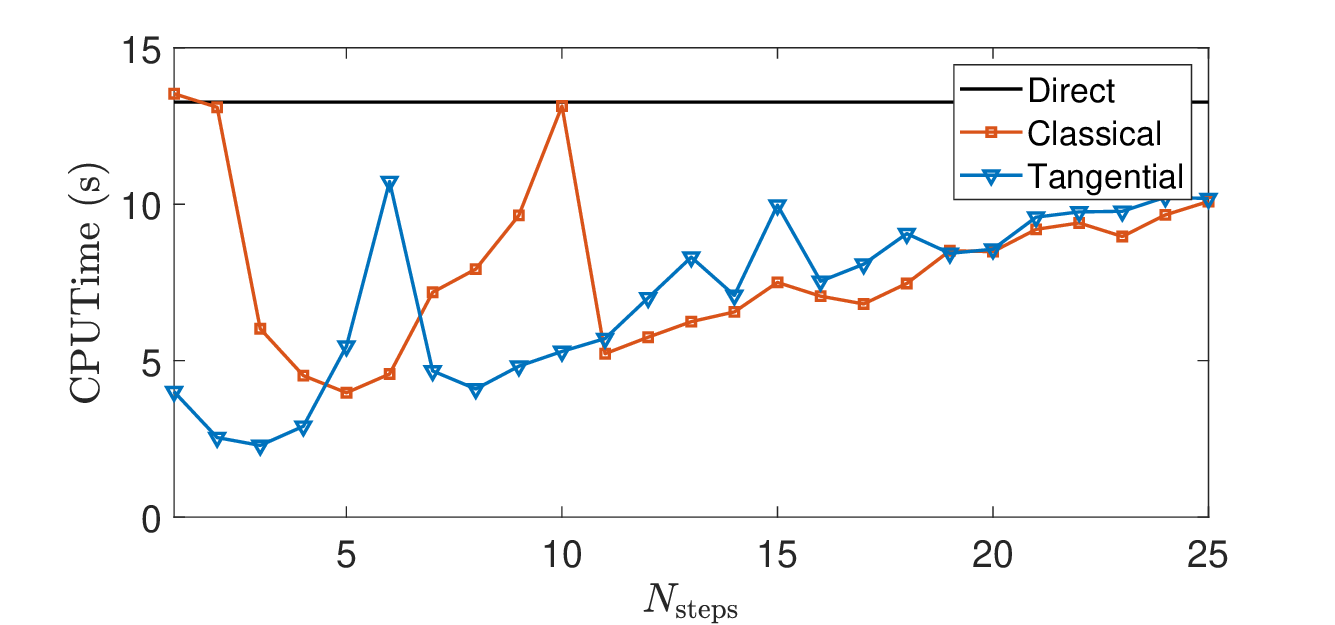}
		\end{minipage}
		\caption{\footnotesize RTR iterations (left) and computation time (right) versus the number of continuation steps for the fixed step size RNC on the matrix completion problem.}
	\end{figure}\begin{figure}                
		\centering 
		\includegraphics[trim=7.8cm 1.5cm 6.6cm 0,clip,width=\textwidth]{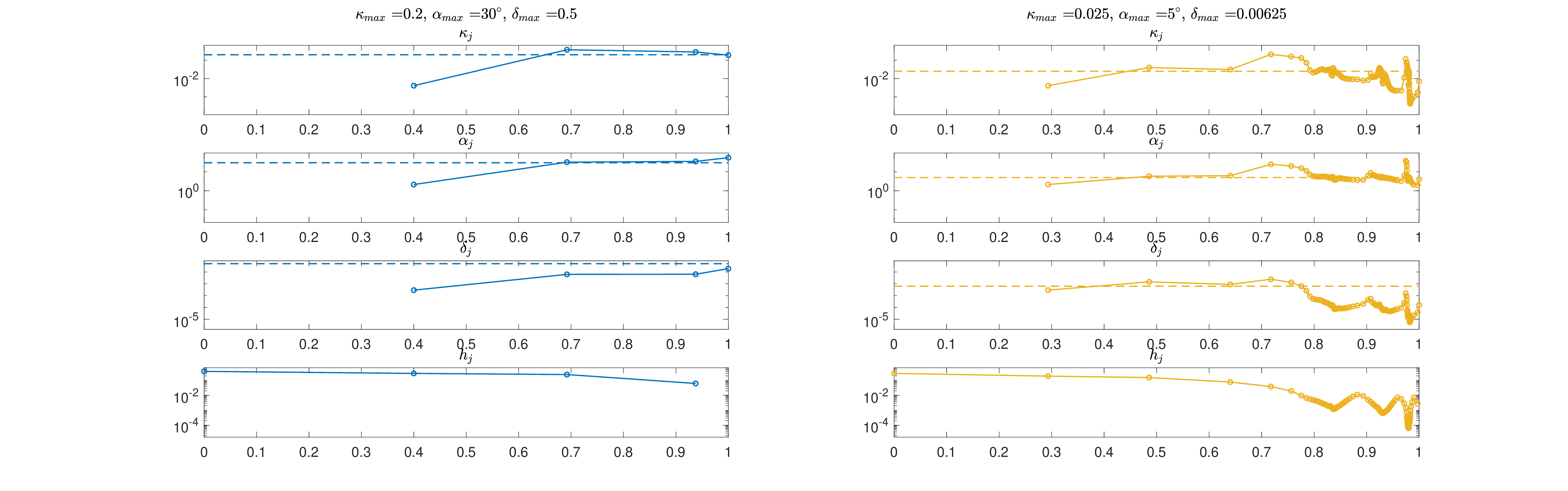}
		\caption{\footnotesize Step size selection on the matrix completion problem. Indicators~\eqref{eq.delta},~\eqref{eq.kappa},~\eqref{eq.alpha} measured after running algorithm~\ref{alg.stepSizeAdaptivity} for selecting the step size (bottom plot), are plotted against the corresponding continuation parameter $\lambda$. The dashed lines are the hyperparameters $\kappa_{\max}$, $\alpha_{\max}$, $\delta_{\max}$ used in the step size adaptivity procedure for each experiment. }
		\label{fig.MatrixCompletionPerformanceIndicators}
	\end{figure}\begin{figure}
		\centering
		\includegraphics[trim=1.75cm 1.2cm 1.75cm 0,clip,width=\textwidth]{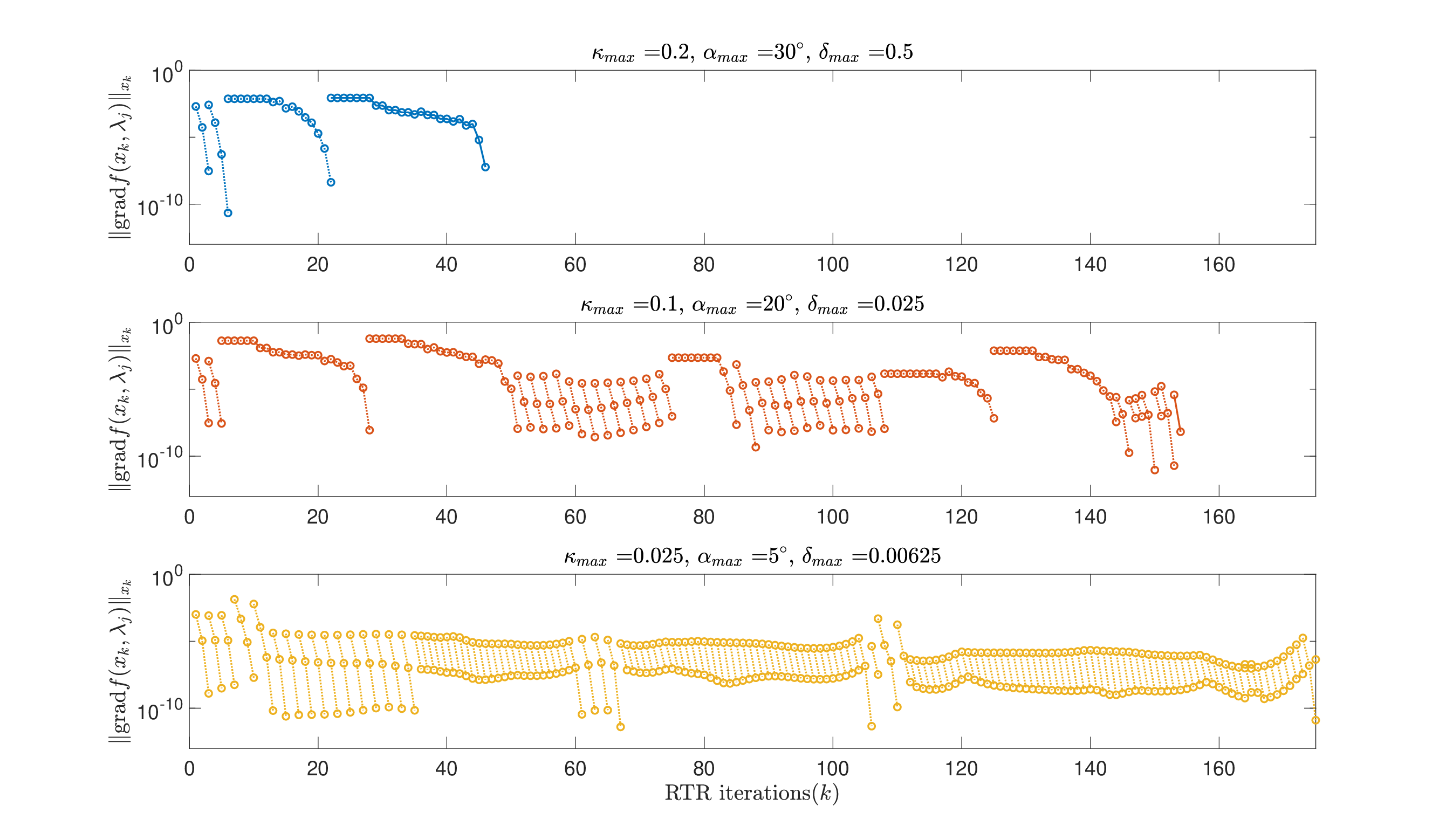}
		\caption{\footnotesize Convergence of the Riemannian gradient norm of each intermediate problem versus RTR iterations on the matrix completion problem. The step size adaptive RNC algorithm is compared for different step size adaptivity hyperparameters. }
		\label{fig.MatrixCompletionStepSizeAdaptive}
	\end{figure} As also confirmed by the step size adaptivity experiments (Figures~\ref{fig.MatrixCompletionPerformanceIndicators} and~\ref{fig.MatrixCompletionStepSizeAdaptive}), the solution curve to the homotopy generated by the instance curve~\eqref{eq.matrixCompletionInstanceCurve} is initially trivial to trace. Indeed, in the first part of the homotopy very few RTR iterations per correction are necessary for the classical prediction and even less for the tangential prediction. We clearly get a sense of the increasing difficulty by observing the results of Figure~\ref{fig.MatrixCompletionPerformanceIndicators}. Shorter and shorter step sizes are chosen in order to satisfy the step size selection criteria. Finally, as seen from the last plot in Figure~\ref{fig.MatrixCompletionStepSizeAdaptive}, completely removing the stagnation from the correction phase requires to enforce very strict step size selection criteria causing very small step sizes to be taken and numerous intermediate corrections to be performed. All in all, the most effective setting is the step size adaptive configuration with a permissive step size selection criteria (first plot in Figure~\ref{fig.MatrixCompletionStepSizeAdaptive}), which still exhibited transient stagnations. We therefore conclude that continuation is effective when the stagnation in the correction is mitigated, while removing completely this behavior requires an effort that is not worthwhile.
	\begin{table}
		\centering
		\caption{\footnotesize Summary of the number of iterations and computation time for the numerical experiments on the matrix completion problem. The parameters for the step size adaptive experiments (1), (2) and (3) are the same as in Figure~\ref{fig.MatrixCompletionStepSizeAdaptive}, from top to bottom. }
		\begin{tabular}{c|c|c|c|}
			\cline{2-4}
			& \multicolumn{3}{c|}{\textbf{Matrix completion}} \\ \hline
			\multicolumn{1}{|c|}{Direct (RTR)} & 1 & 159 & 10.67 \\ \hline
			\multicolumn{1}{|c|}{Direct (LRGeomCG)} & 1 & 1117 & 5.78 \\ \hline
			\multicolumn{1}{|c|}{Direct (LMAFit)} & 1 & 17309 & 15.65 \\ \hline
			\multicolumn{1}{|c|}{Fixed step size classical RNC} & 5 & 68 & 3.39 \\ \hline
			\multicolumn{1}{|c|}{Fixed step size tangential RNC} & 3 & 32 & 1.96 \\ \hline
			\multicolumn{1}{|c|}{Step size adaptive RNC (1)} & 4 & 46 & 4.96 \\ \hline
			\multicolumn{1}{|c|}{Step size adaptive RNC (2)} & 32 & 154 & 70.01 \\ \hline
			\multicolumn{1}{|c|}{Step size adaptive RNC (3)} & 143 & 175 & 259.30 \\ \hline
			& Corrections & Correction iterations & Time (s) \\ \cline{2-4} 
		\end{tabular}
		\label{tab.summaryMatrixCompletionExperiments}
	\end{table}
	
	\section{Conclusions}
	In this work, we have proposed a generalization of numerical continuation to the setting of Riemannian optimization and stated sufficient conditions for the existence of a solution curve. The central contribution is the RNC Algorithm~\ref{alg.riemannianContinuation}, a path-following predictor-corrector algorithm relying on the concept of retraction for the prediction combined with superlinearly converging Riemannian optimization routines such as Riemannian Newton method or the Riemannian Trust Region algorithm for the correction. This method can track a curve of critical points of a parametric Riemannian optimization problem when an initial point on the curve is given. We have generalized to the Riemannian case an adaptive step size strategy relying on the asymptotic expansion of the some performance indicators of the correction. Furthermore, we have provided the analysis of the prediction phase motivating the choice of tangential prediction over classical prediction.
	
	The behavior of our algorithm has been illustrated for the problem of computing the Karcher mean of positive definite matrices and for low-rank matrix completion. Particular homotopies have been proposed for both problems, thereby suggesting a more general approach for achieving this task: defining smooth curves of problem instances starting from an easily solvable one and ending at the instance of interest. This proved to be successful in particular for the matrix completion problem, where a fast decay of singular values leads to a challenging  optimization task.  The step size adaptivity proved to effectively control the Newton update vector norm, the Newton contraction rate and the prediction vectors angle allowing for the correction algorithms to directly exhibit superlinear convergence. However this came at a relatively high computational cost due to the small step sizes required making the fixed step size continuation or permissive step size selection more competitive. 
	
	\addcontentsline{toc}{section}{References}
	\bibliographystyle{alpha}
	\bibliography{bibliographyWithBibTex}
	
	\begin{appendices}
		\section{Proof of lemma~\ref{lem.asymptExpansion}}\label{a:proofLemma3}
		
		Our proof mimics the proof of the Euclidean case from~\cite[Section 6.1]{allg}, making use of a local chart to map the problem to $\Rbb^d$.
		
		We assume that $h$ is sufficiently small such that $x$, $y(h)$ and $z(h)$ are contained in the domain of the same local chart $\pa{\ucal,\varphi}$. All the involved points and functions are expressed in local coordinates~\cite[Chapter 1]{lee} associated with $\varphi$ and the tangent vectors are decomposed into the the local coordinate vector fields~\cite[Exmaple 8.2]{lee} induced by $\varphi$. The same strategy is used in the local convergence proof of the Riemannian Newton in~\cite[Theorem 6.3.2]{absilbook}.  In the following, for the convenience of the reader, we recall how the different entities are mapped to local coordinates; see~\cite{absilbook} for details.
		
		\begin{itemize}
			\item For a point $w\in \ucal$, we write $\hat w := \varphi(w)$. The solution point $x$, the predicted point $y(h)$ and the first RN iterate $z(h)$ become respectively,
			\begin{align*}
				\hx &:= \varphi(x),\\
				\hy(h) &:= \varphi(y(h)),\\
				\hz(h) &:=\varphi(z(h)).
			\end{align*}
			Conversely, for a vector $\hat w \in \hat\ucal:= \varphi(\ucal)$, we write $w:=\varphi^{-1}(\hat w)$. 
			\item For a tangent vector $\xi\in T_w\mcal$ for some $w\in \ucal$, we write $\hat \xi := \D\varphi(w)\pac{\xi}\in\Rbb^d$. The Riemannian gradient $\grad f(w,\lambda)$, its differential with respect to lambda $\ddp{\grad f(w,\lambda)}{\lambda}$, the tangential prediction vector $t(w,\lambda)$ and the RN update vector $n(w,\lambda)$ translate respectively to 
			\begin{align*}
				\hF(\hw,\lambda) &:= \D\varphi(w)\pac{\grad{f}(w,\lambda)},\\
				\hF_\lambda(\hw,\lambda) &:= \D\varphi(w)\pac{\ddp{\grad{f}(w,\lambda)}{\lambda}},\\
				\htt(\hw,\lambda) &:= \D\varphi(w)\pac{t(w,\lambda)},\\
				\hn(\hw,\lambda)&:=\D\varphi(w)\pac{n(w,\lambda)},
			\end{align*}
			for any $(\hw,\lambda)\in\hat \ucal \times \pac{0,1}$. Conversely, given $\hat\xi \in \Rbb^d$ and some $\hw\in\varphi\pa{\ucal}$, we write $\xi := \D\varphi^{-1}(\hw)\pac{\hat \xi}\in T_w\mcal$.
			\item The coordinate representation of the Riemannian Hessian is \begin{equation*}
				\hat H : \hat\ucal\times\pac{0,1}\to \Rbb^{d\times d}: (\hat w,\lambda)\mapsto \D\varphi(w)\pac{\hessx{f}(w,\lambda)\pac{\D\varphi^{-1}(\hat w)\pac{\cdot}}}.
			\end{equation*}
			\item As discussed in~\cite[Chapter 2]{leeRiemManifs}, the Riemannian metric can be represented with the Gramian matrix in the basis of coordinate vector field as  
			\begin{equation*}
				{\hat G_{\hat w}:\hat\ucal\to \Rbb^{d\times d}:\hat w\mapsto \pa{\scalp{\D\varphi^{-1}(\hat w)\pac{e_i}}{\D\varphi^{-1}(\hat w)\pac{e_j}}_w}_{i,j = 1,\dots,d}},
			\end{equation*}where $e_i$ are the canonical vectors of $\Rbb^d$. By smoothness of the Riemannian metric, this function is also smooth. Furthermore, given $\xi\in T_x\mcal$ for some $w\in\ucal$ it holds that $\left\|\xi\right\|_x = \sqrt{\hat\xi^\top \hG_{\hx} \hat\xi}.$
			\item Given a sufficiently small $\xi\in T_w\mcal$ for some $w\in\ucal$, the retraction point $R_w(\xi)$ is well defined and $R_w(\xi)\in\ucal$. For the local representation of such vectors, the coordinate representation of the retraction is  
			\begin{equation*}
				\hat R_{\hat w} (\hat\xi) = \varphi(R_w(\D \varphi^{-1}(\hat w)\pac{\hat\xi})).
			\end{equation*}
			\item Finally, the coordinate representation of the transporter is 
			\begin{equation*}
				\hTcal_{\hy\to \hx}:\Rbb^d\to \Rbb^d: \hat{\xi}\mapsto \D\varphi(x)\pac{\tcal_{y\to x}\pa{\D \varphi^{-1}(\hy)\pac{\hat{\xi}}}},\quad \forall\hy,\hx\in\hUcal.
			\end{equation*}
		\end{itemize}

		Note that the function $\hF_\lambda$ defined above coincides with the derivative of $\hF$ with respect to $\lambda$. However, the differential of $\hF$ with respect to its first argument, denoted $\hF_{\hx}$, does \emph{not} coincide with the $\hH$, the coordinate representation of the Hessian. Indeed, one obtains 
		
		\begin{equation}\label{eq:Fxhat}
			\begin{aligned}
				\hF_\hx(\hw,\lambda)[\cdot] &= \hH(\hw,\lambda)[\cdot] + \D^2\varphi(w)\pac{\grad{f}(w,\lambda),\D\varphi^{-1}(\hat w)\pac{\cdot}}\\
				&=   \hH(\hw,\lambda)[\cdot]  + \D^2\varphi(w)\pac{\D\varphi^{-1}(\hat w)\pac{\hF(\hw,\lambda)},\D\varphi^{-1}(\hat w)\pac{\cdot}} \\
				&= \hH(\hw,\lambda)[\cdot] + \hA(\hw)\pac{\hF(\hw,\lambda),\cdot},
			\end{aligned}
		\end{equation}
		where we defined the bilinear form
		\begin{equation} \label{eq:Ahat}
			\hA(\hw)\pac{\cdot,\cdot} = \D^2\varphi(w)\pac{\D\varphi^{-1}(\hat w)\pac{\cdot},\D\varphi^{-1}(\hat w)\pac{\cdot}}.
		\end{equation}
		On the solution curve, we have $\hF(\hx,\lambda) = 0$, so the second term in~\eqref{eq:Fxhat} vanishes for $\hat w = \hat x$ and we find ${\hF_\hx(\hx,\lambda) = \hH(\hx,\lambda)}$. By the definitions above, the coordinate representations of $y(h)$, $z(h)$, $t(w,\lambda)$ and $n(w, \lambda)$ have the following convenient expressions
		\begin{align}
			\label{eq:yhat}\hy(h) &= \hat R_{\hat x}(h\hat t(\hat x,\lambda)),\\ 
			\label{eq:zhat}\hz(h) &=  \hat R_{\hat y(h)}(\hat n(\hat y(h),\lambda + h)),\\
			\label{eq:that} \hat t(\hat w, \lambda) & =  -\hH(\hat w,\lambda)^{-1}\pac{\hat F_\lambda(\hat w,\lambda)},\\
			\label{eq:nhat}\hat n(\hat w, \lambda) & =  -\hH(\hat w,\lambda)^{-1}\pac{\hat F(\hat w,\lambda)}.
		\end{align}
		As noted in~\cite[Theorem 6.3.2]{absilbook}, we point out that the local rigidity property of the retraction transfers to its local chart version, i.e.
		\begin{equation*}
			\D \hat R_{\hw}(0)\pac{\hat \xi} = \hat \xi,\quad \forall \hw\in\hat \ucal, \:\hat\xi\in\operatorname{dom}(\hR_\hw).
		\end{equation*}
		Using the previous definitions, we conclude the following expressions:
		\begin{align*}
			\delta(x, \lambda,h)\! &=\! \sqrt{\hat n(\hat y(h),\lambda + h)^T \hat G_{\hat y(h)}\hat n(\hat y(h),\lambda + h)},\\
			\kappa(x, \lambda,h)\! &=\! \frac{\sqrt{\hat n(\hat z(h),\lambda + h)^T \hat G_{\hat z(h)}\hat n(\hat z(h),\lambda + h)}}{\delta(x,\lambda,h)},\\
			\alpha(x, \lambda,h)\! &=\! \acos\pa{\!\frac{\htt(\hx,\lambda)^T}{\sqrt{\htt(\hx,\lambda)^T\hG_\hx\htt(\hx,\lambda)}}\hG_{\hx}\frac{\hTcal_{\hy\to \hx}(\htt(\hy,\lambda+h))}{\sqrt{\hTcal_{\hy\to \hx}(\htt(\hy,\lambda+h))^T\hG_\hx\hTcal_{\hy\to \hx}(\htt(\hy,\lambda+h))}}}
		\end{align*}
		We are now in the position to perform Taylor expansion with respect to $h$ of these functions. 
		
		\noindent \emph{Result (i)}\\
		By combining~\eqref{eq:yhat} and~\eqref{eq:nhat} we have 
		\begin{equation} \label{eq:nhatOfY}
			\hn(\hat y(h),\lambda + h) = -\hH(\hat R_{\hat x}(h\hat t(\hat x,\lambda)),\lambda +h)^{-1}\pac{\hat F(\hat R_{\hat x}(h\hat t(\hat x,\lambda)),\lambda+h)}.
		\end{equation}
		Let us expand both terms separately. 
		\begin{equation} \label{eq:FhatExpansion}
			\begin{aligned}
				\hat F(\hat R_{\hat x}(h\hat t(\hat x,\lambda)),\lambda+h) \!&=\! \hat F(\hat x,\lambda)\! + \! h\!\pa{\!\hat F_{\hat x}(\hat x,\lambda)\!\pac{\hat t(\hat x,\lambda)}\! + \!\hat F_\lambda(\hat x,\lambda)} \! + \! h^2 c_1(\hat x, \lambda) \! + \! \Ocal{h^3} \\
				\!&=\! h^2 c_1(\hat x, \lambda) \! + \! \Ocal{h^3},
			\end{aligned}
		\end{equation}
		where the second equality follows from $\hF(\hx,\lambda) = 0$, $\hF_{\hx}(\hx,\lambda) = \hH(\hx,\lambda)$ and~\eqref{eq:that}.
		For later purposes, let us note the explicit expression
		\begin{equation}\label{eq:c1}
			\begin{aligned}
				c_1(\hat x, \lambda) = \frac{1}{2}\Big(\hF_{\hx\hx}(\hx,\lambda)\pac{\htt(\hx,\lambda),\htt(\hx,\lambda)} &+ \hF_\hx(\hx,\lambda)\pac{\D^2\hR_\hx(0)\pac{\htt(\hx,\lambda),\htt(\hx,\lambda)}}  \\  + 2 \hF_{\hx\lambda}(\hat x,\lambda)\pac{\htt(\hx,\lambda)} &+ \hF_{\lambda\lambda}(\hx,\lambda)\Big). 
			\end{aligned}
		\end{equation}
		Now note that 
		\begin{equation*}
			\hH(\hat R_{\hat x}(h\hat t(\hat x,\lambda)),\lambda +h) = \hH(\hat x,\lambda) + \Ocal{h}.
		\end{equation*}
		Then by smoothness of matrix inversion 
		\begin{equation*}
			\hH(\hat R_{\hat x}(h\hat t(\hat x,\lambda)),\lambda +h)^{-1} = \hH(\hat x,\lambda)^{-1}+ \Ocal{h}.
		\end{equation*}
		Combined with~\eqref{eq:FhatExpansion} one has
		\begin{equation}\label{eq:nhatOfYExpansion}
			\hn(\hat y(h),\lambda + h) = h^2 c_2(\hat x, \lambda) + \Ocal{h^{3}}.
		\end{equation}
		with $c_2(\hat x, \lambda) = -\hH(\hat x,\lambda)^{-1}\pac{c_1(\hat x, \lambda)}$. Noting that $	\hG_{\hat y(h)} = \hG_{\hat x} + \Ocal{h}$ we obtain
		\begin{align*}
			\delta(x, \lambda,h) =& \sqrt{\hat n(\hat y(h),\lambda + h)^T \hG_{\hat y(h)}\hat n(\hat y(h),\lambda + h)} \\ =& (h^4c_3(\hx, \lambda)^2 + \Ocal{h^5})^{1/2} \\ =& h^2 c_3(\hx, \lambda) + \Ocal{h^3},
		\end{align*}
		where $c_3(\hx, \lambda) := \sqrt{c_2(\hat x, \lambda)^T \hG_\hx c_2(\hat x, \lambda)}$. The last equality follows from the Taylor expansion of the square root in $c_3(\hx,\lambda)^2$. This is possible provided the $c_3(\hx, \lambda)$ does not vanish. By hypothesis~\eqref{eq.nonDegeneracy}, it can be shown that $c_1(\hx, \lambda)$ is not zero. Hence, $c_2(\hx, \lambda)$ and $c_3(\hx, \lambda)$ are also not zero.
		Setting $\delta_2(x, \lambda) := c_3(\varphi(x), \lambda)$, this concludes the proof of (i). \\
		
		\noindent \emph{Result (ii)}\\
		To obtain the expansion for $\kappa$, we combine result (i) with the expansion of the Newton direction evaluated in $\hz(h)$.
		For this purpose, note that by combining~\eqref{eq:zhat} and~\eqref{eq:nhat}
		\begin{equation} \label{eq:nhatOfZ}
			\hn(\hat z(h),\lambda + h) = -\hH(\hR_{\hy(h)}(\hn(\hy(h),\lambda + h),\lambda +h)^{-1}\pac{\hat F(\hR_{\hy(h)}(\hn(\hy(h),\lambda + h),\lambda+h)}.
		\end{equation}
		The Taylor expansion with respect to $\hn(\hy(h),\lambda + h)$ of the right-hand side term gives
		\begin{align*}
			\hF\pa{\hR_{\hy(h)}(\hn(\hy(h),\lambda + h)),\lambda + h} = \hF(\hy(h),\lambda + h) + \hH(\hy(h),\lambda + h)\pac{\hn(\hy(h),\lambda + h)}\\ +  \hA(\hy(h))\pac{\hF(\hy(h),\lambda + h),\hn(\hy(h),\lambda + h)}\\+ 
			\frac{1}{2}\hF_{\hx\hx}(\hy(h),\lambda + h)\pac{\hn(\hy(h),\lambda + h),\hn(\hy(h),\lambda + h)} \\ 
			+ \frac{1}{2}\hF_\hx(\hy(h),\lambda + h)\!\pac{\D^2\hR_\hx(0)\pac{\hn(\hy(h),\lambda + h),\hn(\hy(h),\lambda + h)}}\! + \Ocal{\|\hn(\hy(h),\lambda + h)\|^3}.
		\end{align*}
		The first two summands cancel out owing to~\eqref{eq:nhat}. Furthermore, by smoothness of the retraction and of the local charts, we have \begin{align*}
			&\hA(\hy(h)) = \hA(\hx) + \Ocal{h},\\
			&\hF_{\hx\hx}(\hy(h),\lambda + h) = \hF_{\hx\hx}(\hat x,\lambda) + \Ocal{h},\\
			&\hF_\hx(\hy(h),\lambda + h)\circ\D^2\hR_{\hy(h)}(0) = \hF_\hx(\hx,\lambda)\circ\D^2\hR_{\hx}(0) + \Ocal{h}.
		\end{align*}
		By plugging in the Taylor expansions of $\hn(\hy(h),\lambda + h)$ and $\hF(\hy(h),\lambda + h)$ given by~\eqref{eq:FhatExpansion} and~\eqref{eq:nhatOfYExpansion} respectively we obtain 
		\begin{equation*}
			\hF(\hz(h),\lambda + h) = h^4c_4(\hx, \lambda) + \Ocal{h^5},
		\end{equation*}
		for some $c_4(\hx, \lambda)$ not depending on $h$. 
		
		Now, for the left-hand side term in~\eqref{eq:nhatOfZ}, the Taylor expansion with respect to $\hn(\hy(h),\lambda + h)$ gives
		\begin{equation*}
			\hH(\hz(h),\lambda + h) = \hH(\hy(h),\lambda + h) + \Ocal{\|\hn(\hy(h),\lambda + h)\|} = \hH(\hx,\lambda) + \Ocal{h},
		\end{equation*}
		and thus 
		\begin{equation*}
			\hH(\hz(h),\lambda + h)^{-1} = \hH(\hx,\lambda)^{-1} + \Ocal{h}.
		\end{equation*}
		Therefore 
		\begin{equation*}
			\hn(\hz(h),\lambda + h) = h^4c_5(\hx, \lambda) + \Ocal{h^5},
		\end{equation*}
		where $c_5(\hx, \lambda) = -\hH(\hx,\lambda)^{-1}\pac{c_4(\hx, \lambda)}$. Finally, noticing that $\hG_{\hz(h)} = \hG_{\hx} + \Ocal{h}$, we can approximate the numerator of $\kappa$ as
		\begin{equation}\label{eq:nHatOfZNormExpansion}
			\sqrt{\hat n(\hat z(h),\lambda\! +\! h)^T G_{\hat z(h)}\hat n(\hat z(h),\lambda \!+\! h)} = \pa{h^8c_6(\hx, \lambda)^2\! +\! \Ocal{h^9}}^{1/2}\! = h^4c_6(\hx, \lambda)\! +\! \ocal{h^4},
		\end{equation}
		with $c_6(\hx, \lambda) = \sqrt{c_5(\hx, \lambda) G_\hx c_5(\hx, \lambda)}$. This allows to conclude that 
		\begin{equation*}
			\kappa(x, \lambda,h) = \frac{h^4c_6(\hx, \lambda) + \ocal{h^4}}{h^2 c_3(\hx, \lambda) + \Ocal{h^3}} = h^2 c_7(\hx, \lambda) + \ocal{h^2},
		\end{equation*}
		with $c_7(\hx, \lambda) = \frac{c_6(\hx, \lambda)}{c_3(\hx, \lambda)}$, where we used the Taylor expansion of the inverse function in $c_3(\hx,\lambda)$, which is non-zero as noted for result (i). This proves the expansion (ii) with ${\kappa_2(x, \lambda) = c_7(\varphi(x), \lambda)}$.
		
		\noindent \emph{Result (iii)}\\
		The proof for the prediction angle requires to expand the argument of the arcosine at second order and exploit the following Puiseux series expansion, for $q>0$:
		\begin{equation}\label{eq.PuiseuxSeries}
			\acos(1-q) = \sqrt{2q} + \frac{q^{3/2}}{6\sqrt{2}} + \Ocal{q^2}.
		\end{equation}
		By combining~\eqref{eq:yhat} and~\eqref{eq:that}, we find
		\begin{equation*}
			\htt(\hy(h),\lambda+h) = -\hH(\hat y(h),\lambda+h)^{-1}\pac{\hat F_\lambda(\hat y(h),\lambda+h)}.
		\end{equation*}

		Concerning the transport of this vector, we exploit the smoothness of $\hy$, $\hH$, $\hF_\lambda$ and of the transporter operator to conclude
		\begin{equation}\label{eq:transportExpansion}
			\hTcal_{\hy(h)\to \hx}\pa{\htt(\hy(h),\lambda + h)} = \htt(\hx,\lambda) + h \hTcal^{(1)}(\hx,\lambda) + h^2\hTcal^{(2)}(\hx,\lambda) + \Ocal{h^3},
		\end{equation}
		for some $\hTcal^{(1)}(\hx,\lambda)$ and $\hTcal^{(2)}(\hx,\lambda)$ depending smoothly only on $(\hx,\lambda)$.
		
		Let us from now on omit the dependence on $(\hx,\lambda)$ of these vectors (i.e. $\htt = \htt(\hx,\lambda)$, $\hTcal^{(1)} = \hTcal^{(1)}(\hx,\lambda)$, $\hTcal^{(2)} = \hTcal^{(2)}(\hx,\lambda)$). Computing the inner product of ~\eqref{eq:transportExpansion} with itself and using the Taylor expansion of the square root in $\|\htt\|^2$ we have
		\begin{equation}\label{eq:transportNormExpansion}
			\!\!\!\|\hTcal_{\hy(h)\to \hx}\!\pa{\htt(\hy(h),\lambda\!+\! h)}\!\|\!=\! \|\htt\| + h\frac{\htt^T\hG_{\hx}\hTcal^{(1)}}{\|\htt\|} + h^2\!\!\pa{\!\frac{2\htt^T\hG_{\hat x}\hTcal^{(2)}\!\!+\!\! \|\hTcal^{(1)}\|^2}{2\|\htt\|}\!-\!\frac{\pa{\!\htt^T\hG_{\hat x}\hTcal^{(1)}\!}\!\!^2}{2\|\htt\|^3}\!}\!.
		\end{equation} 
		Then, combing~\eqref{eq:transportExpansion} and~\eqref{eq:transportNormExpansion} with the expansion of the inverse function we get
		\begin{align*}
			\frac{\hTcal_{\hy(h)\to \hx}\pa{\htt(\hy(h),\lambda + h)}}{\|\hTcal_{\hy(h)\to \hx}\pa{\htt(\hy(h),\lambda + h)}\|} = \frac{\htt}{\|\htt\|} + h\pa{\frac{\hTcal^{(1)}}{\|t\|} - \frac{\htt^T\hG_{\hat x}\hTcal^{(1)}}{\|\htt\|^3}\htt}+\\ \frac{h^2}{2}\pa{\frac{3(\htt^T\hG_{\hat x}\hTcal^{(1)})^2}{\|\htt\|^5}\htt-\frac{2\htt^T\hG_{\hat x}\hTcal^{(2)} + \|\hTcal^{(1)}\|^2}{\|\htt\|^3}\htt - \frac{2\htt^T\hG_\hx\hTcal^{(1)}}{\|\htt\|^3}\hTcal^{(1)} + \frac{2\hTcal^{(2)}}{\|\htt\|}} + \Ocal{h^3}.
		\end{align*}
		Computing the inner product of this expression with $\frac{\htt}{\|\htt\|}$ with respect to the metric $\hG_{\hx}$ we get $\cos\pa{\alpha(x,\lambda, h)}$ and it can be see that the term proportional to $h$ vanishes. Thus if we denote $\cos(\theta_\hx(\hat\xi,\hat\eta)) = \frac{\hat{\xi}^T\hG_\hx\hat{\eta}}{\|\hat{\xi}\|\|\hat{\eta}\|}$ we find 
		\begin{align*}
			\cos(\alpha(x,\lambda)) &= \frac{\htt^T}{\|\htt\|}\hG_\hx\frac{\hTcal_{\hy(h)\to \hx}\pa{\htt(\hy(h),\lambda + h)}}{\|\hTcal_{\hy(h)\to \hx}\pa{\htt(\hy(h),\lambda + h)}\|}  \\
			&= 1 + \frac{h^2}{2}\pa{\frac{\pa{\htt^T\hG_\hx\hTcal^{(1)}}^2}{\|\htt\|^4} - \frac{\|\hTcal^{(1)}\|^2}{\|\htt\|^2}} + \Ocal{h^3}\\
			&= 1-h^2\frac{\sin(\theta_\hx(\htt,\hTcal^{(1)}))^2\|\hTcal^{(1)}\|^2}{2\|\htt\|^2}+\Ocal{h^3}.
		\end{align*}
		Assumptions~\eqref{eq.nonDegeneracy} and~\eqref{eq:missingAssumption} imply that coefficient multiplied by $h^2$ is not zero. Finally, using the Puiseux series~\eqref{eq.PuiseuxSeries} for the arcosine, we conclude
		\begin{equation*}
			\alpha(x,\lambda, h) = h\alpha_1(x, \lambda) + \Ocal{h^2},\text{ with } \alpha_1(x, \lambda) = \frac{|\sin(\theta_\hx(\htt,\hTcal^{(1)}))|\|\hTcal^{(1)}\|}{\|\hat t\|}.
		\end{equation*}
		
		\flushright$\square$
		
	\end{appendices}
	
\end{document}